\documentclass[12pt]{article}

\usepackage{amsthm,tikz,amssymb,latexsym,amsmath,amscd,amsfonts,array, graphicx,lmodern,slantsc,enumerate,stmaryrd,rotating}

\usepackage{hyperref}
\hypersetup{nesting=true,debug=true,naturalnames=true}
\usepackage[all]{xy}
\CompileMatrices

\usepackage{color}

\newtheorem{proposition}{Proposition}[section]
\newtheorem{corollary}[proposition]{Corollary}
\newtheorem{definition}[proposition]{Definition}
\newtheorem{theorem}[proposition]{Theorem}
\newtheorem{remark}[proposition]{Remark}
\newtheorem{example}[proposition]{Example}
\newtheorem{lemma}[proposition]{Lemma}

\newtheorem{problem}[proposition]{Problem}

\begin{document}

\title{On the homotopy type of complexes of graphs with bounded domination number}
\author{Jes\'us Gonz\'alez and Teresa~I.~Hoekstra-Mendoza}
\date{}

\maketitle

\begin{abstract}
Let $D_{n,\gamma}$ be the complex of graphs on $n$ vertices and domination number at least~$\gamma$. We prove that $D_{n,n-2}$ has the homotopy type of a finite wedge of 2-spheres. This is done by using discrete Morse theory techniques. Acyclicity of the needed matching is proved by introducing a relativized form of a well known method for constructing acyclic matchings on suitable chunks of simplices. Our approach allows us to extend our results to the realm of infinite graphs. In addition, we give evidence supporting the assertion that the homotopy equivalences $D_{n,n-1}\simeq \bigvee S^0$ and $D_{n,n-2}\simeq \bigvee S^2$ do not seem to generalize for $D_{n,\gamma}$ with $\gamma\leq n-3$.
\end{abstract}

{\small 2010 Mathematics Subject Classification: Primary: 55P15. Secondary: 05C69, 55U10, 57M15.}

{\small Keywords and phrases: Complex of a graph property, domination number, discrete Morse theory, acyclic matching, homotopy type.}

\section{Introduction}
For an integer $n\geq2$, let $\mathcal{G}_n$ stand for the family of simple, undirected graphs on vertices $\overline{n}=\{1,2,\ldots,n\}$, and let $\mathfrak{P}$ be a graph property in $\mathcal{G}_n$, i.e., $\mathfrak{P}$ is a subset of $\mathcal{G}_n$ that is closed under isomorphism classes. It is customary to use the expression ``$\sigma$ satisfies $\mathfrak{P}$'' as a replacement for ``$\sigma\in\mathfrak{P}$''. The property $\mathfrak{P}$ is said to be monotone provided it is satisfied by any graph in $\mathcal{G}_n$ that is obtained by removing an edge from a graph already satisfying~$\mathfrak{P}$. Since the vertex set is fixed, we may identify each graph $\sigma\in\mathcal{G}_n$ with its edge set. In these terms, a monotone property $\mathfrak{P}$ gives rise to an abstract simplicial complex $K(\mathfrak{P})$. Explicitly, $d$-dimensional simplices of $K(\mathfrak{P})$ are given by the graphs having $d+1$ edges that satisfy $\mathfrak{P}$.

Complexes of graphs, i.e., the geometric realizations of abstract simplicial complexes arising from monotone graph properties, is the subject of an active and fruitful line of research in combinatorial topology. For one, complexes of graphs have deep connections to developments in other fields. For instance, the complex defined by the family of $n$-vertex matchings plays a subtle role in Ksontini's investigation (\cite{MR1996066,MR2095571}) of the homotopy type of the Quillen complex associated to the lattice of subgroups (at a given prime) of the $n$-symmetric group. Likewise, complexes of (not $i$-) connected graphs appear in Vassiliev's study of the cohomology groups of spaces of knots (\cite{MR1738399}), while complexes of bounded-degree graphs can be used in the analysis of homological properties of free two-step nilpotent Lie algebras (\cite{MR923671, MR1409979}). In addition, complexes of graphs are intrinsically rich objects whose properties have been investigated by a number of authors. The reader is referred to Jonsson's book~\cite{MR2368284} for details and a complete list of references.

In this paper we focus on the complex $D_{n,\gamma}$ defined by $n$-vertex graphs whose domination number is at least $\gamma$ (see Section~\ref{antecedentes} for background details). Our main result shows that $D_{n,n-2}$ is homotopy equivalent to a wedge of 2-spheres:

\begin{theorem}\label{teoremaprincipal}
For $n\ge4$, let
\begin{equation}\label{nsw}
N_n=\frac1{12}(n-2)(n-3)(3n^2-7n-2).
\end{equation}
Then $D_{n,n-2}$ has the homotopy type of $\,\bigvee_{N_n} S^2$, a wedge of $N_n$ 2-dimensional spheres.
\end{theorem}

The situation for $n=3$ could be thought of as a special instance in Theorem~\ref{teoremaprincipal}: as observed in the next section, $D_{3,1}$ is contractible (and $N_3=0$). On the other hand, our proof method (discrete Morse theory, reviewed in Section~\ref{antecedentes}) allows us to show that the homotopy equivalences $D_{n,n-2}\simeq\bigvee_{N_n} S^2$ fit together as $n$ varies:

\begin{theorem}\label{teoremaprincipalextendido}
The union $D^2:=\bigcup_{n\geq3}D_{n,n-2}$ has the homotopy type of a countable wedge of 2-spheres.
\end{theorem}

A combinatorial interpretation of $D^2$ as a complex of graphs on countably-many vertices, as well as the description of inclusions $D_{n,n-2}\hookrightarrow D_{n+1,n-1}$, are deferred to Section~\ref{graficasinfinitas} where Theorems~\ref{teoremaprincipal} and~\ref{teoremaprincipalextendido} are proved. 

\medskip
Our methods allow us to identify (in Section~\ref{52}) the homotopy type of $D_{5,2}$:
\begin{proposition}\label{d52}
The graph complex $D_{5,2}$ is homotopy equivalent to a wedge of four spheres of dimension 5.
\end{proposition}

Theorem~\ref{teoremaprincipal}, Corollary~\ref{teoremaprincipalextendido} and Proposition~\ref{d52} are similar to results in the literature where particular families of complexes of graphs are shown to be homotopy equivalent to wedges of copies of some fixed sphere (see for instance~\cite{MR1660341}). It would thus be reasonable to ask whether the three statements above are special cases of a result on the same vein that applies to any $D_{n,n-k}$. However, the statement of such a result ---if it exists--- cannot be a straightforward extension of Theorem~\ref{teoremaprincipal} and Corollary~\ref{d52}. For instance, $D_{6,3}$ does not have the homotopy type of a wedge of \emph{odd} dimensional spheres (not even if we allow odd spheres of different dimensions). Indeed, computer calculations report that the Euler characteristic of $D_{6,3}$ is $92$. So, if this space splits up to homotopy as a wedge of spheres, then many of the splitting wedge sphere summands would have to be even-dimensional (to allow for a positive Euler characteristic). This motivates the following questions:

\begin{problem}{\em
Is it true that all complexes $D_{n,k}$ have the homotopy type of a wedge of spheres? Is there a complex $D_{n,k}$ which has the homotopy type of a wedge of spheres not all of which have the same dimension?
}\end{problem}

\section{Graph theoretic preliminaries}\label{antecedentes}
We start by reviewing standard definitions in graph theory. Recall $n\geq2$. A set $D\subseteq\overline{n}$ dominates a graph $\sigma\in\mathcal{G}_n$ if every vertex in $\overline{n}-D$ is $\sigma$-adjacent to some vertex in $D$ (thus, such a $D$ must be non-empty). The domination number of $\sigma\in\mathcal{G}_n$, denoted by $\gamma(\sigma)$, is the minimal cardinality of vertex sets that dominate~$\sigma$. Thus $1\leq\gamma(\sigma)\leq n$.

\begin{remark}\label{partcs}{\em
The following observations follow directly from the definition: The graph $\varnothing\in\mathcal{G}_n$ with no edges has $\gamma(\varnothing)=n$, while any other graph $\sigma\in\mathcal{G}_n$ has $1\leq\gamma(\sigma)\leq n-1$. In fact, a graph with a single edge has domination number equal to $n-1$, while the condition $\gamma(\sigma)\leq n-2$ is forced as soon as $\sigma$ has at least two edges. Actually the equality $\gamma(\sigma)=n-2$ holds if $\sigma$ has exactly two edges. (The relationship between domination number and edge cardinality is much subtler than what is apparent from the last two observations. For instance, the graph $\sigma\in\mathcal{G}_n$ with the three edges $\{1,2\}, \{2,3\}$ and $\{3,4\}$ has $\gamma(\sigma)=n-2$.)
}\end{remark}

Fix $k\in\{1,2,\ldots,n\}$. Any vertex set dominating a graph $\tau$ will also dominate any graph $\sigma$ containing $\tau$ as a subgraph, so that $\gamma(\tau)\ge\gamma(\sigma)$. Consequently, the family $\mathfrak{P}_k$ of graphs $\sigma\in\mathcal{G}_n$ satisfying $\gamma(\sigma)\geq k$ is a monotone property. We are interested in the homotopy properties of the (geometric realization of the) associated graph complex $$D_{n,k}:=K(\mathfrak{P}_k).$$ Explicitly, $d$-dimensional simplices of $D_{n,k}$ are given by the graphs in $\mathcal{G}_n$ having $d+1$ edges and domination number at least $k$. By Remark~\ref{partcs}, $D_{n,n}$ is empty, while $D_{n,1}=\Delta^{\genfrac(){0pt}{5}{n}{2} -1}$, the full complex on $\genfrac(){0pt}{1}{n}{2}$ vertices (i.e., the complex where any subset of vertices forms a simplex). So, to avoid trivial cases, we restrict attention to $D_{n,k}$ for $1<k<n$, in which case (again by Remark~\ref{partcs}) the vertex set of $D_{n,k}$ is the family of $\genfrac(){0pt}{1}{n}{2}$ edges of the complete graph $K_n$. Then, a family~$E$ of $d+1$ edges forms a $d$-simplex of $D_{n,k}$ provided the graph $\sigma\in\mathcal{G}_n$ with edge set $E$
has $\gamma(\sigma)\ge k$. In particular, the facets of $D_{n,k}$ (i.e., the simplices of $D_{n,k}$ that are maximal with respect to inclusion) are the graphs $\sigma\in\mathcal{G}_n$ which are maximal with respect to the condition $\gamma(\sigma)\geq k$. For instance, $D_{4,2}$ has two types of facets: (\emph{i}) cycles on four vertices and (\emph{ii}) graphs of $\mathcal{G}_4$ with two components, one of which is an isolated vertex and the other is a complete graph on three vertices (see Example~\ref{d42} in the next section). The latter facets are 2-dimensional while the former ones have maximal possible dimension 3. More generally, facets of $D_{n,k}$ of maximal possible dimension were described by Vizing in terms of minimal edge covers:

\begin{definition}
An edge cover of a graph $G$ is a set $C$ of edges of $G$ such that each vertex of $G$ is incident with at least one edge of $C$. A minimal edge cover is an edge cover of smallest possible cardinality.
\end{definition}

\begin{theorem}[Vizing~\cite{MR0188106}]\label{vizing}
Let $1<k<n$. The dimension of the complex $D_{n,k}$ is one less than the integral part of $$\frac12(n-k+2)(n-k).$$ Facets of $D_{n,k}$ of dimension $\dim(D_{n,k})$ are the graphs of the form $K'_{n-k+2}+(k-2)$. Here and below the notation $G+(m)$ stands for the graph obtained by adding~$m$ isolated vertices to a graph $G$, and $K'_m$ stands for a graph obtained from the complete graph on $m$ vertices $K_m$ by removing the edges in some minimal edge cover of $K_m$.
\end{theorem}

By Remark~\ref{partcs}, $D_{n,n-1}$ is a wedge of 0-spheres ---the zero-th skeleton of $\Delta^{\genfrac(){0pt}{5}{n}{2} -1}$. Our main result, Theorem~\ref{teoremaprincipal}, gives a similar description of the homotopy type of the first interesting case in the filtration of subcomplexes
\begin{equation}\label{filtration}
\xymatrix{D_{n,n-1}\;\ar@{^(->}[r] & \,D_{n,n-2}\; \ar@{^(->}[r] & \,D_{n,n-3}\; \ar@{^(->}[r] &\cdots\; \ar@{^(->}[r] & \,D_{n,2}\; \ar@{^(->}[r] & \,D_{n,1}\, \ar@{=}[r] & \,\Delta^{\genfrac(){0pt}{5}{n}{2} -1}}.
\end{equation}
On the other hand, as far as we are aware of, the homotopy type of $D_{n,k}$ is unknown for $n-3\geq k\geq 2$. In the introduction, we have made a point in this regard in the case of $D_{n,n-3}$. Here we address the case of $D_{n,2}$, which turns out to be an instance of a family of much studied complexes associated to a different (but related) graph property. Namely, consider the complex $D^d_n$ associated to the monotone property on $\mathcal{G}_n$ consisting of the graphs $\sigma$ all whose vertices have degree at most~$d$ (i.e., no vertex of $\sigma$ is adjacent to more than $d$ edges). The reader is referred to~\cite{MR2022345} for a review of the known properties of these complexes. Here we start by noting that $D_{n,k}\subseteq D_n^{n-k}$. Indeed, a graph $\sigma\in\mathcal{G}_n$ not representing a simplex of $D^{n-k}_d$ must have a vertex $v$ of degree $\deg(v)>n-k$. Then $v$ and the $n-\deg(v)-1$ vertices not adjacent to $v$ form a set of cardinality $n-\deg(v)$ that dominate $\sigma$, so that
$ 
\gamma(\sigma)\leq n-\deg(v)<k,
$ 
and $\sigma$ does not represent a simplex of $D_{n,k}$ either. Therefore~(\ref{filtration}) extends to
\begin{equation}\label{filtracion}
\begin{gathered}\xymatrix{
D_{n,n-1}\;\ar@{^(->}[r] \ar@{^(->}[d] & \,D_{n,n-2}\; \ar@{^(->}[r] \ar@{^(->}[d] & \cdots\; \ar@{^(->}[r] & \,D_{n,2}\; \ar@{^(->}[r] \ar@{^(->}[d] & \,D_{n,1}\, \ar@{=}[r] \ar@{^(->}[d] & \,\Delta^{\genfrac(){0pt}{5}{n}{2} -1} \\
D_{n}^1\;\ar@{^(->}[r] & \,D_{n}^{2}\; \ar@{^(->}[r] & \cdots\; \ar@{^(->}[r] & \,D_{n}^{n-2}\; \ar@{^(->}[r] & \,D_{n}^{n-1}. & 
}\end{gathered}
\end{equation}
Now, by definition, $D_n^{n-1}=\Delta^{\genfrac(){0pt}{5}{n}{2} -1}=D_{n,1},$ the largest complex in~(\ref{filtration}). Our point is that the next-to-the-right-most vertical inclusion in~(\ref{filtracion}) is an equality too. Indeed, if $\tau\in\mathcal{G}_n$ is a simplex of $D_n^{n-2}$, then no vertex of $\tau$ can have degree $n-1$, so that $\gamma(\tau)\geq2$ and, thus, $\tau$ is a simplex of $D_{n,2}$.\footnote{The agreement between $D_{n,k}$ and $D^{n-k}_n$ for $k=1,2$ fades away as $k$ grows. For instance, $D_{n,n-1}$ is 0-dimensional, while the reduced homology groups $\widetilde{H}^*(D^1_n;\mathbb{Q})$ are non-zero in a range of dimensions that grows with $n$ (see~\cite{MR1174893}).} We remark that $D_n^{n-2}$ has been studied by Jonsson (see Section~18.3 in~\cite{MR2368284}, particularly Corollary~18.20) via its Alexander dual, which is the complex associated to graphs in $\mathcal{G}_n$ having an isolated vertex.

\section{Discrete Morse theory}\label{secndmt}
Theorem~\ref{teoremaprincipal} will be a direct consequence of the main theorem in discrete Morse theory (Theorem~\ref{DMTkey} below) once we construct an acyclic matching on $D_{n,n-2}$ with $N_n$ critical 2-simplices, 1 critical 0-simplex, and no other critical simplices. This section is intended to record the needed basic background material. Details can be found in standard references such as~\cite{MR1612391,MR3379451}.

Le $X$ be an abstract simplicial complex with face poset $\mathcal{F}$, i.e., $\mathcal{F}$ is the set of simplicies of $X$ partially ordered by inclusion. Here and below, for a simplex $\alpha\in\mathcal{F}$, we will write $\alpha^{(p)}$ to indicate that $\alpha$ is $p$-dimensional. The Hasse diagram of $\mathcal{F}$, denoted by $H_\mathcal{F}$, is the directed graph with vertex set $\mathcal{F}$ and edges given by the family of ordered pairs $(\alpha^{(p+1)},\beta^{(p)})$ with $\beta\subset\alpha$. Such an edge will also be denoted as $\alpha^{(p+1)}\searrow\beta^{(p)}$. For a matching $\mathcal{P}$ on $H_\mathcal{F}$ (i.e.~a directed subgraph of $H_\mathcal{F}$ whose vertices have degree at most 1), the modified Hasse diagram $H_\mathcal{F}(\mathcal{P})$ is the directed graph obtained from $H_\mathcal{F}$ by reversing all the the matching edges, i.e., the edges of $\mathcal{P}$. A matching edge will be denoted as $\beta^{(p)}\nearrow\alpha^{(p+1)}$ and, in such a case, $\alpha$ is said to be $\mathcal{P}$-collapsible, while $\beta$ is said to be $\mathcal{P}$-redundant. In these terms, a directed path in $H_\mathcal{F}(\mathcal{P})$ is spelled out by an alternate sequence of up-going and down-going arrows:
\begin{equation}\label{path}
\tau_0\nearrow\sigma_1\searrow\tau_1\nearrow\cdots\nearrow\sigma_k\searrow\tau_k.
\end{equation}
The path in~(\ref{path}) is said to
\begin{itemize}
\item[(\emph{a})] have length $k$;
\item[(\emph{b})] be simple if $\tau_i\neq\tau_j$ for $i\neq j$;
\item[(\emph{c})] be a cycle if $k>0$, $\tau_0=\tau_k$ and the set $\{\tau_\ell\colon1\le\ell\le k\}$ has cardinality $k$.
\end{itemize}
Items ({\emph{a}}) and ({\emph{b}}) make sense (with $k=\infty$) for infinite paths $\tau_0\nearrow\sigma_1\searrow\tau_1\nearrow\cdots$. The matching $\mathcal{P}$ is said to be acyclic if $H_\mathcal{F}(\mathcal{P})$ has no directed cycles. Simplices of $X$ that are neither $\mathcal{P}$-redundant nor $\mathcal{P}$-collapsible are said to be $\mathcal{P}$-critical.

\begin{theorem}[{See~\cite[Corollary 3.5]{MR1612391}}]\label{DMTkey}
The geometric realization of a finite simplicial complex $X$ with an acyclic matching $\mathcal{P}$ has the homotopy type of a CW complex having one cell of dimension $d$ for each $\mathcal{P}$-critical $d$-simplex of $X$.
\end{theorem}

Acyclic matchings on a simplicial complex $X$ can be constructed by dividing the family of simplices of $X$ into smaller families on each of which (typically more manageable) matchings are to be defined. The formulation we need, Proposition~\ref{descomposicion} below, is a special instance of~\cite[Lemma~4.1]{MR2164921}.

\begin{proposition}\label{descomposicion}
Let the simplices of a finite simplicial complex $X$ be partitioned into pairwise disjoint families of simplices $X_\ell$, $\ell=1,\ldots,k$. Assume:
\begin{itemize}
\item[(i)] $X_1\cup\cdots\cup X_\ell$ is a subcomplex of $X$, for $\ell=1,2,\ldots, k-1$.
\item[(ii)] There is an acyclic matching $\mathcal{P}_\ell$ on $X_\ell$, for $\ell=1,2,\ldots, k$. 
\end{itemize}
Then $\mathcal{P}:=\bigcup_{1\le\ell\le k} \mathcal{P}_\ell$ is an acyclic matching on $X$.
\end{proposition}

We will also need the following relativized form of Proposition~\ref{descomposicion}:
\begin{proposition}\label{relativized}
Let $X$, $X_\ell$ and $\mathcal{P}_\ell$ be as in Proposition~\ref{descomposicion}. Assume $X$ is a subcomplex of a larger complex $Y$ whose simplices are also partitioned into pairwise disjoint families of simplices $Y_\ell$, $\ell=1,\ldots,k$, satisfying the following three conditions:
\begin{itemize}
\item[(i)] $Y_\ell\cap X=X_\ell$, for $\ell=1,2,\ldots,k$.
\item[(ii)] $Y_1\cup\cdots\cup Y_\ell$ is a subcomplex of $Y$, for $\ell=1,2,\ldots,k-1$.
\item[(iii)] For $\ell=1,2,\ldots,k$, there is a matching $\mathcal{Q}_\ell$ on $Y_\ell$ (no assumption is made about acyclicity of $\mathcal{Q}_\ell$) which restricts to $\mathcal{P}_\ell$ on $X_\ell$, in the sense that (\emph{a}) two simplices of $X_\ell$ form a $\mathcal{P}_\ell$-matched pair if and only if they form a $\mathcal{Q}_\ell$-matched pair, and that (\emph{b})~no simplex of $X_\ell$ can be $\mathcal{Q}_\ell$-matched to a simplex of $Y_\ell-X_\ell$.
\end{itemize}
Assume in addition that, for each $\ell=1,2,\cdots,k$, there are no directed $\mathcal{Q}_\ell$-cycles in $Y_\ell-X_\ell$. Then each $\mathcal{Q}_\ell$ is in fact acyclic, so that (by Proposition~\ref{descomposicion}) $\mathcal{Q}:=\bigcup_{1\le\ell\le k} \mathcal{Q}_\ell$ is an acyclic matching on $Y$.
\end{proposition}
\begin{proof}
Let $\mathcal{F}_X$ and $\mathcal{F}_Y$ stand for the face posets of $X$ and $Y$, respectively. Consider a directed edge $\beta\nearrow\alpha$ of $H_{\mathcal{F}_Y}(\mathcal{Q})$. Condition {\em (iii)} means that, if either $\alpha$ or $\beta$ belongs to $\mathcal{F}_X$, then in fact both $\alpha$ and $\beta$ belong to $\mathcal{F}_X$, so that in fact $\beta\nearrow\alpha$ is a directed edge of $H_{\mathcal{F}_X}(\mathcal{P})$. Since $X$ is a subcomplex of $Y$, condition {\em (i)} now yields that any directed $\mathcal{Q}_\ell$-cycle would have to either be a directed $\mathcal{P}_\ell$-cycle or, else, remain entirely outside $X_\ell$. Since both alternatives are ruled out by hypothesis, the proof is complete.
\end{proof}

The proof of Theorem~\ref{teoremaprincipalextendido} (given in Section~\ref{graficasinfinitas}) uses the following extension of Theorem~\ref{DMTkey} to infinite complexes:

\begin{theorem}[{\cite[Theorem~20]{MR3049255}}]\label{DMTkeyinfinito}
Let $\mathcal{P}$ be an acyclic matching on a (possibly infinite) simplicial complex $X$ with face poset $\mathcal{F}$. If the modified Hasse diagram $H_\mathcal{F}(\mathcal{P})$ contains no infinite directed simple paths, then the geometric realization of $X$ has the homotopy type of a CW complex having one cell of dimension $d$ for each $\mathcal{P}$-critical $d$-simplex of $X$.
\end{theorem}

\section{Structure of \texorpdfstring{$D_{n,n-2}$}{Dnn2}}\label{structuresection}
The hypothesis $n\geq4$ will be in force from this point on. By Remark~\ref{partcs}, $D_{n,n-2}$ has all possible simplices in dimensions $0$ and $1$. On the other hand, by Theorem~\ref{vizing}, $D_{n,n-2}$ has dimension 3, with 3-dimensional facets given by all 4-cycles with $n-4$ additional isolated vertices. In particular, the 2-dimensional simplices of $D_{n,n-2}$ that are not facets are the graphs of the form
\begin{equation}\label{endim2}
\begin{tikzpicture}[x=.6cm,y=.6cm]
\draw(1,1.5)--(2,1.5);\draw(2,1.5)--(3,1.5);\draw(3,1.5)--(4,1.5);
\node at (1,1.5) {\scriptsize$\bullet$};\node at (2,1.5) {\scriptsize$\bullet$};
\node at (3,1.5) {\scriptsize$\bullet$};\node at (4,1.5) {\scriptsize$\bullet$};
\node at (1,1.5) {\scriptsize$\bullet$};\node at (2,1.5) {\scriptsize$\bullet$};
\node at (5.5,1.5) {\footnotesize{${}+(n-4)$,}};
\end{tikzpicture}
\end{equation}
The remaining simplices of $D_{n,n-2}$ are described next.

\begin{lemma}\label{combinatorics2} 
The 2-dimensional facets of $D_{n,n-2}$ are the graphs of the form $K_3+(n-3)$, where $K_3$ stands for a complete graph on three vertices.
\end{lemma}
\begin{proof}
As noted in Remark~\ref{partcs}, a graph of the form $K_3+(n-3)$ has domination number $n-2$; furthermore, it is maximal with respect to the latter condition. The only other (not yet considered) possibilities for a 2-dimensional graph are
\begin{equation*}
\begin{tikzpicture}[x=.6cm,y=.6cm]
\draw(1,1)--(2,1);\draw(1,2)--(2,2);\draw(1,1.5)--(2,1.5);
\node at (1,2) {\scriptsize$\bullet$};\node at (1,1) {\scriptsize$\bullet$};
\node at (2,2) {\scriptsize$\bullet$};\node at (2,1) {\scriptsize$\bullet$};
\node at (1,1.5) {\scriptsize$\bullet$};\node at (2,1.5) {\scriptsize$\bullet$};
\node at (3.5,1.5) {\footnotesize{${}+(n-6)$,}};

\draw(9,1)--(10,1);\draw(9,2)--(10,2);\draw(9,2)--(8,2);
\node at (9,2) {\scriptsize$\bullet$};\node at (9,1) {\scriptsize$\bullet$};\node at (8,2) {\scriptsize$\bullet$};
\node at (10,2) {\scriptsize$\bullet$};\node at (10,1) {\scriptsize$\bullet$};
\node at (11.5,1.5) {\footnotesize{${}+(n-5)$,}};

\draw(17,2)--(18,2);\draw(17,2)--(16,2);\draw(17,2)--(17,1);
\node at (17,2) {\scriptsize$\bullet$};\node at (16,2) {\scriptsize$\bullet$};
\node at (18,2) {\scriptsize$\bullet$};\node at (17,1) {\scriptsize$\bullet$};
\node at (19.5,1.5) {\footnotesize{${}+(n-4)$,}};
\end{tikzpicture}
\end{equation*}
but they have domination number $n-3$.
\end{proof}

In what follows, the edge between vertices $i$ and $j$ ($i,j\in\overline{n}$) is denoted by~$ij$ (no distinction is made between $ij$ and $ji$). In addition, for edges $a, b, c,\ldots$, we write $a|b|c|\cdots$ as a shorthand for the graph of $\mathcal{G}_n$ whose edge set is $\{a,b,c,\cdots\}$.

\begin{example}\label{d42}{\em
Facets of $D_{4,2}$ are the 3-simplexes $12|13|24|34$, $12|14|23|34$ and $13|14|23|24$, and the 2-simplexes $12|13|23$, $12|14|24$, $13|14|34$ and $23|24|34$. 
}\end{example}

\section{The matching}\label{section3}
Recall $D_{n,1}=\Delta^{\genfrac(){0pt}{5}{n}{2} -1}$. Let:
\begin{itemize}
\item $\mathcal{P}\hspace{.3mm}'_{\hspace{-.3mm}12}$ be the matching on $D_{n,1}$ given by inclusion-exclusion of the edge $12$;
\item $\mathcal{P}_{12}$ be the restriction of $\mathcal{P}\hspace{.3mm}'_{\hspace{-.3mm}12}$ to $D_{n,n-2}$;
\item $X_{12}$ be the subset of $D_{n,n-2}$ consisting of
\begin{itemize}
\item[(i)] the graph with the single edge $12$, and 
\item[(ii)] the graphs with a $\mathcal{P}_{12}$-matching pair.
\end{itemize}
\end{itemize}

In what follows, for $\sigma\in\mathcal{G}_n$, we write $\sigma+ij$ as a substitute for $\sigma\cup\{ij\}$. The following follows directly from the construction:
\begin{proposition}\label{x12iscomplex}
$X_{12}=\{\sigma\in D_{n,n-2}\colon\sigma+12\in D_{n,n-2}\}$, which is a subcomplex of $D_{n,n-2}$.
\end{proposition}

Note that $\mathcal{P}_{12}$ is an acyclic matching in $X_{12}$ (because $\mathcal{P}\hspace{.3mm}'_{\hspace{-.3mm}12}$ is so in $D_{n,1}$) with the simplex in (i) above as its only critical simplex in dimension 0, and with no critical simplices in positive dimensions. The rest of the section is devoted to prove:

\begin{proposition}\label{summa}
There is a matching $\mathcal{Q}$ on the simplices
\begin{equation}\label{r12}
R_{12}:=D_{n,n-2}-X_{12}=\{\sigma\in D_{n,n-2}\colon \sigma+12\not\in D_{n,n-2}\}
\end{equation}
all whose critical (i.e.~unpaired) simplices are 2-dimensional.
\end{proposition}

\begin{corollary}\label{estesieselbueno}
$\mathcal{D}_{n,n-2}:=\mathcal{P}_{12}\cup\mathcal{Q}$ is a matching in $D_{n,n-2}$ with a single critical simplex in dimension 0, and all other critical simplices being 2-dimensional. 
\end{corollary}

\begin{remark}\label{not12}{\em
Note that $12\not\in\sigma$, for all $\sigma\in R_{12}$. Further, by Remark~\ref{partcs}, $X_{12}$ contains the 0-skeleton of $D_{n,n-2}$, so that $R_{12}$ contains simplices only in dimensions in between 1 and 3. 
}\end{remark}

The portion of $\mathcal{Q}$ that matches simplices of dimension 1 with simplices of dimension 2, denoted by $\mathcal{Q}_1^2$, is specified in Definition~\ref{q12} below using:

\begin{proposition}\label{lexi}
Let $\sigma$ be a 1-dimensional cell in $R_{12}$. The set $M_{\sigma}$ consisting of the edges $ij$ satisfying $i<j<n$ and $\sigma+ij\in R_{12}$ is nonempty. 
\end{proposition}

\begin{definition}\label{q12}
Taking the $(i,j)$-lexicographic order on edges $ij$ with $i<j$, $\mathcal{Q}_1^2$ matches a 1-dimensional cell $\sigma\in R_{12}$ with $\sigma+ij$, where $ij$ is the first edge in $M_\sigma$. 
\end{definition}

\begin{example}\label{ejed42}{\em
For $n=4$, the only 1-dimensional simplices in $R_{12}$ are
\begin{equation*}
\begin{tikzpicture}[x=.6cm,y=.6cm]
\draw(9,1)--(9,2);\draw(9,2)--(10,1);
\node at (9,2) {\scriptsize$\bullet$};\node at (8.6,2) {\scriptsize$1$};
\node at (9,1) {\scriptsize$\bullet$};\node at (8.6,1) {\scriptsize$3$};
\node at (10,2) {\scriptsize$\bullet$};\node at (10.4,2) {\scriptsize$2$};
\node at (10,1) {\scriptsize$\bullet$};\node at (10.4,1) {\scriptsize$4$};

\draw(15,1)--(15,2);\draw(14,1)--(15,2);
\node at (14,2) {\scriptsize$\bullet$};\node at (13.6,2) {\scriptsize$1$};
\node at (14,1) {\scriptsize$\bullet$};\node at (13.6,1) {\scriptsize$3$};
\node at (15,2) {\scriptsize$\bullet$};\node at (15.4,2) {\scriptsize$2$};
\node at (15,1) {\scriptsize$\bullet$};\node at (15.4,1) {\scriptsize$4$};
\end{tikzpicture}
\end{equation*}
and their respective $\mathcal{Q}_1^2$ pairs are
\begin{equation*}
\begin{tikzpicture}[x=.6cm,y=.6cm]
\draw(9,1)--(9,2);\draw(9,2)--(10,1);\draw(10,2)--(9,1);
\node at (9,2) {\scriptsize$\bullet$};\node at (8.6,2) {\scriptsize$1$};
\node at (9,1) {\scriptsize$\bullet$};\node at (8.6,1) {\scriptsize$3$};
\node at (10,2) {\scriptsize$\bullet$};\node at (10.4,2) {\scriptsize$2$};
\node at (10,1) {\scriptsize$\bullet$};\node at (10.4,1) {\scriptsize$4$};

\draw(15,1)--(15,2);\draw(14,1)--(15,2);\draw(14,2)--(14,1);
\node at (14,2) {\scriptsize$\bullet$};\node at (13.6,2) {\scriptsize$1$};
\node at (14,1) {\scriptsize$\bullet$};\node at (13.6,1) {\scriptsize$3$};
\node at (15,2) {\scriptsize$\bullet$};\node at (15.4,2) {\scriptsize$2$};
\node at (15,1) {\scriptsize$\bullet$};\node at (15.4,1) {\scriptsize$4$};
\end{tikzpicture}
\end{equation*}
}\end{example}

The graphs in Example~\ref{ejed42} are linear, a situation that holds in general:

\begin{proof}[Proof of Proposition~\ref{lexi}]
We consider the only two possible forms the simplex $\sigma$ can have.

\medskip\noindent
{\bf Case I.} Assume $\sigma$ has the form
\begin{equation}\label{opt1}
\raisebox{-4.5mm}{
\begin{tikzpicture}[x=.6cm,y=.6cm]
\draw(11,1)--(10,1);\draw(10,2)--(11,2);
\node at (11,2) {\scriptsize$\bullet$};\node [above] at (11,2) {\scriptsize $b$};
\node at (11,1) {\scriptsize$\bullet$};\node [above] at (10,2) {\scriptsize $a$};
\node at (10,1) {\scriptsize$\bullet$};\node [below] at (10,1) {\scriptsize $c$};
\node at (10,2) {\scriptsize$\bullet$};\node [below] at (11,1) {\scriptsize $d$};
\node at (12.5,1.5) {\footnotesize{${}+(n-4)$,}};
\end{tikzpicture}
}
\end{equation}
where $ab\neq12\neq cd$. Without loss of generality, we can assume $a<b$, $c<d$ and $a<c$. By~(\ref{endim2}) and~(\ref{r12}), we have $\sigma+ac\in D_{n,n-2}$ with $ac\neq12$. (The latter condition explains why this case is not part of Example~\ref{ejed42}.) Further, the characterization of 3-dimensional simplices in $D_{n,n-2}$ coming from Vizig's Theorem~\ref{vizing} gives $\sigma+ac+12\notin D_{n,n-2}$, so that~(\ref{r12}) gives $\sigma+ac\in R_{12}$. Thus $ac\in M_\sigma$. In fact, $ac$ is the first element in $M_\sigma$, for the characterization of 2-dimensional simplices in $D_{n,n-2}$ coming from~(\ref{endim2}) and Lemma~\ref{combinatorics2} yields $M_\sigma\subseteq\{ac,ad,bc,bd\}$. Thus, in this case we have the $\mathcal{Q}_1^2$ matching
\begin{equation}\label{qmatch1}\raisebox{-.7cm}{
\begin{tikzpicture}[x=.6cm,y=.6cm]
\draw(11,-.5)--(10,-.5);\draw(10,.5)--(11,.5);
\node at (11,.5) {\scriptsize$\bullet$};\node [above] at (11,.5) {\scriptsize $b$};
\node at (11,-.5) {\scriptsize$\bullet$};\node [above] at (10,.5) {\scriptsize $a$};
\node at (10,-.5) {\scriptsize$\bullet$};\node [below] at (10,-.5) {\scriptsize $c$};
\node at (10,.5) {\scriptsize$\bullet$};\node [below] at (11,-.5) {\scriptsize $d$};
\node at (12.7,0) {\footnotesize{${}+(n-4)$}};

\draw(20.5,-.5)--(19.5,-.5);\draw(20.5,.5)--(19.5,.5);\draw(19.5,-.5)--(19.5,.5);
\node at (20.5,.5) {\scriptsize$\bullet$};\node [above] at (20.5,.5) {\scriptsize $b$};
\node at (20.5,-.5) {\scriptsize$\bullet$};\node [above] at (19.5,.5) {\scriptsize $a$};
\node at (19.5,-.5) {\scriptsize$\bullet$};\node [below] at (19.5,-.5) {\scriptsize $c$};
\node at (19.5,.5) {\scriptsize$\bullet$};\node [below] at (20.5,-.5) {\scriptsize $d$};
\node at (22,0) {\footnotesize{${}+(n-4).$}};
\draw[thick,->](16,-.5)--(17,.5);
\end{tikzpicture}}
\end{equation}

\medskip\noindent
{\bf Case II.} Assume $\sigma$ has the form
\begin{equation}\label{opt2}
\raisebox{-6.4mm}{
\begin{tikzpicture}[x=.6cm,y=.6cm]
\draw(2,1.5)--(3,1.5);\draw(3,1.5)--(3,.5);
\node at (2,1.5) {\scriptsize$\bullet$};\node [above] at (2,1.5) {\scriptsize $a$};
\node at (3,1.5) {\scriptsize$\bullet$};\node [above] at (3,1.5) {\scriptsize $b$};
\node at (3,.5) {\scriptsize$\bullet$};\node [below] at (3,.5) {\scriptsize $c$};
\node at (2,1.5) {\scriptsize$\bullet$};
\node at (4.4,1) {\footnotesize{\ \ ${}+(n-3)$,}};
\end{tikzpicture}
}
\end{equation}
where $ab\neq12\neq bc$. Without loss of generality we can assume $a<c$. If $a=1$ (so $b\neq2$), then $c\neq2$ (in view of~(\ref{r12}) and Lemma~\ref{combinatorics2}), so $\sigma+12\in D_{n,n-2}$ in view of~(\ref{endim2}), which contradicts~(\ref{r12}). Likewise, the equality $a=2$ cannot hold, and we actually have $3\le a<c$. Let $d=2$ ($d=1$) if $b=1$ ($b>1$). Using once again~(\ref{r12}) and the characterization of 2-dimensional simplices in $D_{n,n-2}$ coming from~(\ref{endim2}) and Lemma~\ref{combinatorics2}, we see that $\sigma+da\in R_{12}$ and, in fact, that $da$ is the first element in $M_\sigma$. So we have the $\mathcal{Q}_1^2$ matching
\begin{equation}\label{qmatch2}\raisebox{-.7cm}{
\begin{tikzpicture}[x=.6cm,y=.6cm]
\draw(11,-.5)--(11,.5);\draw(10,.5)--(11,.5);
\node at (11,.5) {\scriptsize$\bullet$};\node [above] at (11,.5) {\scriptsize $b$};
\node at (11,-.5) {\scriptsize$\bullet$};\node [above] at (10,.5) {\scriptsize $a$};
\node at (10,.5) {\scriptsize$\bullet$};\node [below] at (11,-.5) {\scriptsize $c$};
\node at (12.5,0) {\footnotesize{${}+(n-4)$}};

\draw(20.5,-.5)--(20.5,.5);\draw(20.5,.5)--(19.5,.5);\draw(19.5,-.5)--(19.5,.5);
\node at (20.5,.5) {\scriptsize$\bullet$};\node [above] at (20.5,.5) {\scriptsize $b$};
\node at (20.5,-.5) {\scriptsize$\bullet$};\node [above] at (19.5,.5) {\scriptsize $a$};
\node at (19.5,-.5) {\scriptsize$\bullet$};\node [below] at (19.5,-.5) {\scriptsize $d$};
\node at (19.5,.5) {\scriptsize$\bullet$};\node [below] at (20.5,-.5) {\scriptsize $c$};
\node at (22,0) {\footnotesize{${}+(n-4)$}};
\draw[thick,->](16,-.5)--(17,.5);
\end{tikzpicture}
}
\end{equation}
In both cases above, the condition ``$i<j<n$'' in the statement of the proposition holds by construction.
\end{proof}

The previous proof makes use of the characterization of $d$-dimensional simplices of $D_{n,n-2}$ discussed in Section~\ref{structuresection}, and we have carefully poinpointed the characterization case (either $d=2$ or $d=3$) needed at each step of the argument. Some of the arguments below will also make use of these characterizations, and we will make free use of them.

\smallskip
Next we define $\mathcal{Q}_2^3$, the pairing $\mathcal{Q}$ we need in $R_{12}$ between simplices of dimension 2 and simplices of dimension 3. This time we define the matching pair of any 3-simplex $\sigma\in R_{12}$. In short, the rule (in Definition~\ref{q23} below) is that the $\mathcal{Q}_2^3$ matching pair of a 3-simplex $\sigma\in R_{12}$ is obtained by removing the first edge of $\sigma$ (in the lexicographic order). We need:

\begin{proposition}\label{lexi2}
Let $\sigma=a_0b_0|a_1b_1|a_2b_2|a_3b_3$ be a 3-simplex in $R_{12}$, with $a_i<b_i$ for $0\le i\le3$, and $(a_0,b_0)<(a_1,b_1)<(a_2,b_2)<(a_3,b_3)$ ordered lexicographically. Then $a_1b_1|a_2b_2|a_3b_3\!:$
\begin{enumerate}
\item is a 2-simplex in $R_{12}$ which is not paired under $\mathcal{Q}_1^2$;
\item determines $\sigma$; indeed, the latter is the only 3-cell in $R_{12}$ (actually in $D_{n,n-2}$) having the former as a face, i.e.~$a_1b_1|a_2b_2|a_3b_3$ is a free face of $a_0b_0|a_1b_1|a_2b_2|a_3b_3$.
\end{enumerate}
\end{proposition}

\begin{definition}\label{q23}
Under the conditions in Proposition~\ref{lexi2}, the $\mathcal{Q}_2^3$ matching pair of $\sigma$ is defined to be the simplex $\,a_1b_1|a_2b_2|a_3b_3$.
\end{definition}

\begin{proof}[Proof of Proposition~\ref{lexi2}]
By Vizing's Theorem~\ref{vizing}, $\sigma$ has the form
$$
\begin{tikzpicture}[x=.6cm,y=.6cm]
\draw(15.5,-.5)--(15.5,.5);\draw(15.5,.5)--(14.5,.5);\draw(14.5,-.5)--(14.5,.5);\draw(14.5,-.5)--(15.5,-.5);
\node at (15.5,.5) {\scriptsize$\bullet$};\node [above] at (15.5,.5) {\scriptsize $b$};
\node at (15.5,-.5) {\scriptsize$\bullet$};\node [above] at (14.5,.5) {\scriptsize $a$};
\node at (14.5,-.5) {\scriptsize$\bullet$};\node [below] at (14.5,-.5) {\scriptsize $c$};
\node at (14.5,.5) {\scriptsize$\bullet$};\node [below] at (15.5,-.5) {\scriptsize $d$};
\node at (17,0) {\footnotesize{${}+(n-4)$,}};
\end{tikzpicture}
$$
where we can safely assume $a<\min\{b,c,d\}$ and $b<c$. Under these conditions $ab$ plays the role of $a_0b_0$ in the statement of the proposition, and $a_1b_1|a_2b_2|a_3b_3$ becomes
\begin{equation}\label{q23match}
\raisebox{-7mm}{
\begin{tikzpicture}[x=.6cm,y=.6cm]
\draw(15.5,-.5)--(15.5,.5);\draw(14.5,-.5)--(14.5,.5);\draw(14.5,-.5)--(15.5,-.5);
\node at (15.5,.5) {\scriptsize$\bullet$};\node [above] at (15.5,.5) {\scriptsize $b$};
\node at (15.5,-.5) {\scriptsize$\bullet$};\node [above] at (14.5,.5) {\scriptsize $a$};
\node at (14.5,-.5) {\scriptsize$\bullet$};\node [below] at (14.5,-.5) {\scriptsize $c$};
\node at (14.5,.5) {\scriptsize$\bullet$};\node [below] at (15.5,-.5) {\scriptsize $d$};
\node at (17,0) {\footnotesize{${}+(n-4)$.}};
\end{tikzpicture}}
\end{equation}
Note that the latter 2-simplex lies in $R_{12}$ and determines $\sigma$ (as indicated in the statement of the proposition) in view of Vizing's theorem and Remark~\ref{not12}. To complete the proof, it suffices to check that~(\ref{q23match}) cannot appear as the higher dimensional simplex in~(\ref{qmatch1}) or in~(\ref{qmatch2}). For this, note that the degree of the vertex with the smallest label in the non-trivial component of~(\ref{q23match}) ---i.e., vertex $a$--- equals~1. This immediately rules out the case of~(\ref{qmatch1}), as well as the case of~(\ref{qmatch2}) when $b=1$ (this ``$b$'' is used in the context of the notation of~(\ref{qmatch2})). To rule out the remaining case, i.e.~the case of~(\ref{qmatch2}) where (its) vertex $b$ is greater than 1 (so the corresponding label $d$ is 1), it suffices to compare the labels of the two vertices:
\begin{itemize}
\item[(i)] vertex of degree 1 with the higher label (this is vertex $b$ in~(\ref{q23match}), and vertex $c$ in~(\ref{qmatch2}));
\item[(ii)] vertex of degree 2 which is adjacent to the vertex of degree 1 with the smaller label (this is vertex $c$ in~(\ref{q23match}), and vertex $a$ in~(\ref{qmatch2})).
\end{itemize}
For, in the case of~(\ref{q23match}), the label of vertex in (i) is smaller than the label of the vertex in~(ii), whereas the opposite inequality holds in the remaining case of~(\ref{qmatch2}).
\end{proof}

\begin{proof}[Proof of Proposition~\ref{summa}]
In view of Proposition~\ref{lexi2}, it only remains to prove that if two 1-dimensional cells $\sigma,\sigma'\in R_{12}$ are $\mathcal{Q}$-paired to a common 2-dimensional cell in $R_{12}$, then in fact $\sigma=\sigma'$. We consider all possible cases arising from the combination of the forms ((\ref{opt1}) or~(\ref{opt2})) of $\sigma$ and~$\sigma'$.

\medskip\noindent
{\bf Case I}. Assume $\sigma$ and $\sigma'$ have, respectively, the forms
\begin{equation*}
\begin{tikzpicture}[x=.6cm,y=.6cm]
\draw(2,2)--(2,1);\draw(1,2)--(2,2);
\node [above] at (2,2) {\scriptsize $b$};\node [above] at (1,2) {\scriptsize $a$};\node [below] at (2,1) {\scriptsize $c$};
\node at (1,2) {\scriptsize$\bullet$};
\node at (2,2) {\scriptsize$\bullet$};\node at (2,1) {\scriptsize$\bullet$};
\node at (3.5,1.5) {\footnotesize{${}+(n-3)$}};

\node at (6.6,1.5) {and};

\draw(9,1)--(10,1);\draw(9,2)--(10,2);
\node at (9,2) {\scriptsize$\bullet$};\node at (9,1) {\scriptsize$\bullet$};
\node at (10,2) {\scriptsize$\bullet$};\node at (10,1) {\scriptsize$\bullet$};
\node [above] at (9.1,2) {\scriptsize $a'$};
\node [above] at (10.1,2) {\scriptsize $b'$};
\node [below] at (9,1) {\scriptsize $c'$};
\node [below] at (10.1,1) {\scriptsize $d'$};
\node at (11.5,1.5) {\footnotesize{${}+(n-4)$}};
\end{tikzpicture}
\end{equation*}
with $3\leq a<c$, $a'<b'$, $a'<c'<d'$ and, additionally, $b=1$. Then the common $\mathcal{Q}$-matched pair would be
\begin{equation*}
\begin{tikzpicture}[x=.6cm,y=.6cm]
\draw(2,2)--(2,1);\draw(1,2)--(2,2);\draw(1,1)--(1,2);
\node [above] at (2,2) {\scriptsize $1$};\node [above] at (1,2) {\scriptsize $a$};
\node [below] at (2,1) {\scriptsize $c$};\node [below] at (1,1) {\scriptsize $2$};
\node at (1,2) {\scriptsize$\bullet$};\node at (1,1) {\scriptsize$\bullet$};
\node at (2,2) {\scriptsize$\bullet$};\node at (2,1) {\scriptsize$\bullet$};
\node at (3.5,1.5) {\footnotesize{${}+(n-4)$}};

\node at (6.8,1.5) {$=$};

\draw(9,1)--(10,1);\draw(9,2)--(10,2);\draw(9,2)--(9,1);
\node at (9,2) {\scriptsize$\bullet$};\node at (9,1) {\scriptsize$\bullet$};
\node at (10,2) {\scriptsize$\bullet$};\node at (10,1) {\scriptsize$\bullet$};
\node [above] at (9.1,2) {\scriptsize $a'$};
\node [above] at (10.1,2) {\scriptsize $b'$};
\node [below] at (9,1) {\scriptsize $c'$};
\node [below] at (10.1,1) {\scriptsize $d'$};
\node at (11.5,1.5) {\footnotesize{${}+(n-4)$,}};
\end{tikzpicture}
\end{equation*}
giving:
\begin{itemize}
\item $\{1,a\}=\{a',c'\}$, which is possible only with $a'=1$ and $a=c'$ (as $a'<c'$);
\item $\{2,c\}=\{b',d'\}$, which is possible only with $2=d'$ and $c=b'$ (as $a'=1$ and $12\not\in\sigma'$).
\end{itemize}
But then $1=a'<c'<d'=2$, which is impossible.

\medskip\noindent
{\bf Case II}. Assume $\sigma$ and $\sigma'$ have, respectively, the forms
\begin{equation*}
\begin{tikzpicture}[x=.6cm,y=.6cm]
\draw(2,2)--(2,1);\draw(1,2)--(2,2);
\node [above] at (2,2) {\scriptsize $b$};\node [above] at (1,2) {\scriptsize $a$};\node [below] at (2,1) {\scriptsize $c$};
\node at (1,2) {\scriptsize$\bullet$};
\node at (2,2) {\scriptsize$\bullet$};\node at (2,1) {\scriptsize$\bullet$};
\node at (3.5,1.5) {\footnotesize{${}+(n-3)$}};

\node at (6.6,1.5) {and};

\draw(9,1)--(10,1);\draw(9,2)--(10,2);
\node at (9,2) {\scriptsize$\bullet$};\node at (9,1) {\scriptsize$\bullet$};
\node at (10,2) {\scriptsize$\bullet$};\node at (10,1) {\scriptsize$\bullet$};
\node [above] at (9.1,2) {\scriptsize $a'$};
\node [above] at (10.1,2) {\scriptsize $b'$};
\node [below] at (9,1) {\scriptsize $c'$};
\node [below] at (10.1,1) {\scriptsize $d'$};
\node at (11.5,1.5) {\footnotesize{${}+(n-4)$}};
\end{tikzpicture}
\end{equation*}
with $3\leq a<c$, $a'<b'$, $a'<c'<d'$ and, additionally, $b>1$. Then the common $\mathcal{Q}$-matched pair would be
\begin{equation*}
\begin{tikzpicture}[x=.6cm,y=.6cm]
\draw(2,2)--(2,1);\draw(1,2)--(2,2);\draw(1,1)--(1,2);
\node [above] at (2,2) {\scriptsize $b$};\node [above] at (1,2) {\scriptsize $a$};
\node [below] at (2,1) {\scriptsize $c$};\node [below] at (1,1) {\scriptsize $1$};
\node at (1,2) {\scriptsize$\bullet$};\node at (1,1) {\scriptsize$\bullet$};
\node at (2,2) {\scriptsize$\bullet$};\node at (2,1) {\scriptsize$\bullet$};
\node at (3.5,1.5) {\footnotesize{${}+(n-4)$}};

\node at (6.8,1.5) {$=$};

\draw(9,1)--(10,1);\draw(9,2)--(10,2);\draw(9,2)--(9,1);
\node at (9,2) {\scriptsize$\bullet$};\node at (9,1) {\scriptsize$\bullet$};
\node at (10,2) {\scriptsize$\bullet$};\node at (10,1) {\scriptsize$\bullet$};
\node [above] at (9.1,2) {\scriptsize $a'$};
\node [above] at (10.1,2) {\scriptsize $b'$};
\node [below] at (9,1) {\scriptsize $c'$};
\node [below] at (10.1,1) {\scriptsize $d'$};
\node at (11.5,1.5) {\footnotesize{${}+(n-4)$,}};
\end{tikzpicture}
\end{equation*}
giving $1\in\{1,c\}=\{b',d'\}$, which is impossible as both $b'$ and $d'$ are at least 2.

\medskip\noindent
{\bf Case III} Assume $\sigma$ and $\sigma'$ have, respectively, the forms
\begin{equation*}
\begin{tikzpicture}[x=.6cm,y=.6cm]
\draw(1,1)--(2,1);\draw(1,2)--(2,2);
\node [above] at (2,2) {\scriptsize $b$};\node [above] at (1,2) {\scriptsize $a$};
\node [below] at (1,1) {\scriptsize $c$};\node [below] at (2,1) {\scriptsize $d$};
\node at (1,2) {\scriptsize$\bullet$};\node at (1,1) {\scriptsize $\bullet$};
\node at (2,2) {\scriptsize$\bullet$};\node at (2,1) {\scriptsize$\bullet$};
\node at (3.5,1.5) {\footnotesize{${}+(n-4)$}};

\node at (6.6,1.5) {and};

\draw(9,1)--(10,1);\draw(9,2)--(10,2);
\node at (9,2) {\scriptsize$\bullet$};\node at (9,1) {\scriptsize$\bullet$};
\node at (10,2) {\scriptsize$\bullet$};\node at (10,1) {\scriptsize$\bullet$};
\node [above] at (9.1,2) {\scriptsize $a'$};
\node [above] at (10.1,2) {\scriptsize $b'$};
\node [below] at (9,1) {\scriptsize $c'$};
\node [below] at (10.1,1) {\scriptsize $d'$};
\node at (11.5,1.5) {\footnotesize{${}+(n-4)$}};
\end{tikzpicture}
\end{equation*}
with $a<b$, $a<c<d$ and $a'<b'$, $a'<c'<d'$. Then the common $\mathcal{Q}$-matched pair would be
\begin{equation*}
\begin{tikzpicture}[x=.6cm,y=.6cm]
\draw(1,1)--(2,1);\draw(1,2)--(2,2);\draw(1,1)--(1,2);
\node [above] at (2,2) {\scriptsize $b$};\node [above] at (1,2) {\scriptsize $a$};
\node [below] at (2,1) {\scriptsize $d$};\node [below] at (1,1) {\scriptsize $c$};
\node at (1,2) {\scriptsize$\bullet$};\node at (1,1) {\scriptsize$\bullet$};
\node at (2,2) {\scriptsize$\bullet$};\node at (2,1) {\scriptsize$\bullet$};
\node at (3.5,1.5) {\footnotesize{${}+(n-4)$}};

\node at (6.8,1.5) {$=$};

\draw(9,1)--(10,1);\draw(9,2)--(10,2);\draw(9,2)--(9,1);
\node at (9,2) {\scriptsize$\bullet$};\node at (9,1) {\scriptsize$\bullet$};
\node at (10,2) {\scriptsize$\bullet$};\node at (10,1) {\scriptsize$\bullet$};
\node [above] at (9.1,2) {\scriptsize $a'$};
\node [above] at (10.1,2) {\scriptsize $b'$};
\node [below] at (9,1) {\scriptsize $c'$};
\node [below] at (10.1,1) {\scriptsize $d'$};
\node at (11.5,1.5) {\footnotesize{${}+(n-4)$,}};
\end{tikzpicture}
\end{equation*}
giving $b=b'$ or $b=d'$. Either way we get $\sigma=\sigma'$, for
\begin{itemize}
\item $b=b'\Rightarrow\left( \hspace{.3mm}a=a', \; c=c' \mbox{ and } d=d'\,\right) \Rightarrow\sigma=\sigma'$.
\item $b=d'\Rightarrow\left( \hspace{.3mm}a=c', \; c=a' \mbox{ and } d=b'\,\right) \Rightarrow\sigma=\sigma'$.
\end{itemize}

\smallskip\noindent
{\bf Case IV}. Assume $\sigma$ and $\sigma'$ have, respectively, the forms
\begin{equation*}
\begin{tikzpicture}[x=.6cm,y=.6cm]
\draw(2,2)--(2,1);\draw(1,2)--(2,2);
\node [above] at (2,2) {\scriptsize $b$};\node [above] at (1,2) {\scriptsize $a$};\node [below] at (2,1) {\scriptsize $c$};
\node at (1,2) {\scriptsize$\bullet$};
\node at (2,2) {\scriptsize$\bullet$};\node at (2,1) {\scriptsize$\bullet$};
\node at (3.5,1.5) {\footnotesize{${}+(n-3)$}};

\node at (6.6,1.5) {and};

\draw(10,2)--(10,1);\draw(9,2)--(10,2);
\node at (9,2) {\scriptsize$\bullet$};
\node at (10,2) {\scriptsize$\bullet$};\node at (10,1) {\scriptsize$\bullet$};
\node [above] at (9.1,2) {\scriptsize $a'$};
\node [above] at (10.1,2) {\scriptsize $b'$};
\node [below] at (10.1,1) {\scriptsize $c'$};
\node at (11.5,1.5) {\footnotesize{${}+(n-3)$}};
\end{tikzpicture}
\end{equation*}
with $3\leq a<c$, $3\leq a'<c'$ and, additionally, $b=b'$. Then the common $\mathcal{Q}$-matched pair would be
\begin{equation*}
\begin{tikzpicture}[x=.6cm,y=.6cm]
\draw(2,2)--(2,1);\draw(1,2)--(2,2);\draw(1,1)--(1,2);
\node [above] at (2,2) {\scriptsize $b$};\node [above] at (1,2) {\scriptsize $a$};
\node [below] at (2,1) {\scriptsize $c$};\node [below] at (1,1) {\scriptsize $d$};
\node at (1,2) {\scriptsize$\bullet$};\node at (1,1) {\scriptsize$\bullet$};
\node at (2,2) {\scriptsize$\bullet$};\node at (2,1) {\scriptsize$\bullet$};
\node at (3.5,1.5) {\footnotesize{${}+(n-4)$}};

\node at (6.8,1.5) {$=$};

\draw(10,2)--(10,1);\draw(9,2)--(10,2);\draw(9,2)--(9,1);
\node at (9,2) {\scriptsize$\bullet$};\node at (9,1) {\scriptsize$\bullet$};
\node at (10,2) {\scriptsize$\bullet$};\node at (10,1) {\scriptsize$\bullet$};
\node [above] at (9.1,2) {\scriptsize $a'$};
\node [above] at (10.1,2) {\scriptsize $b'$};
\node [below] at (9.1,1) {\scriptsize $d'$};
\node [below] at (10.1,1) {\scriptsize $c'$};
\node at (11.5,1.5) {\footnotesize{${}+(n-4)$,}};
\end{tikzpicture}
\end{equation*}
with $d=d'\in\{1,2\}$, which implies $\sigma=\sigma'$.

\medskip\noindent
{\bf Case V}. Assume $\sigma$ and $\sigma'$ have, respectively, the forms
\begin{equation*}
\begin{tikzpicture}[x=.6cm,y=.6cm]
\draw(2,2)--(2,1);\draw(1,2)--(2,2);
\node [above] at (2,2) {\scriptsize $b$};\node [above] at (1,2) {\scriptsize $a$};\node [below] at (2,1) {\scriptsize $c$};
\node at (1,2) {\scriptsize$\bullet$};
\node at (2,2) {\scriptsize$\bullet$};\node at (2,1) {\scriptsize$\bullet$};
\node at (3.5,1.5) {\footnotesize{${}+(n-3)$}};

\node at (6.6,1.5) {and};

\draw(10,2)--(10,1);\draw(9,2)--(10,2);
\node at (9,2) {\scriptsize$\bullet$};
\node at (10,2) {\scriptsize$\bullet$};\node at (10,1) {\scriptsize$\bullet$};
\node [above] at (9.1,2) {\scriptsize $a'$};
\node [above] at (10.1,2) {\scriptsize $b'$};
\node [below] at (10.1,1) {\scriptsize $c'$};
\node at (11.5,1.5) {\footnotesize{${}+(n-3)$}};
\end{tikzpicture}
\end{equation*}
with $3\leq a<c$, $3\leq a'<c'$ and, additionally, $b>1=b'$. Then the common $\mathcal{Q}$-matched pair would be
\begin{equation*}
\begin{tikzpicture}[x=.6cm,y=.6cm]
\draw(2,2)--(2,1);\draw(1,2)--(2,2);\draw(1,1)--(1,2);
\node [above] at (2,2) {\scriptsize $b$};\node [above] at (1,2) {\scriptsize $a$};
\node [below] at (2,1) {\scriptsize $c$};\node [below] at (1,1) {\scriptsize $1$};
\node at (1,2) {\scriptsize$\bullet$};\node at (1,1) {\scriptsize$\bullet$};
\node at (2,2) {\scriptsize$\bullet$};\node at (2,1) {\scriptsize$\bullet$};
\node at (3.5,1.5) {\footnotesize{${}+(n-4)$}};

\node at (6.8,1.5) {$=$};

\draw(10,2)--(10,1);\draw(9,2)--(10,2);\draw(9,2)--(9,1);
\node at (9,2) {\scriptsize$\bullet$};\node at (9,1) {\scriptsize$\bullet$};
\node at (10,2) {\scriptsize$\bullet$};\node at (10,1) {\scriptsize$\bullet$};
\node [above] at (9.1,2) {\scriptsize $a'$};
\node [above] at (10.1,2) {\scriptsize $b'$};
\node [below] at (9,1) {\scriptsize $2$};
\node [below] at (10.1,1) {\scriptsize $c'$};
\node at (11.5,1.5) {\footnotesize{${}+(n-4)$,}};
\end{tikzpicture}
\end{equation*}
which is impossible, as both $c$ and $c'$ are greater than 3.
\end{proof}

\section{Acyclicity}\label{aciclicidad}
This section is devoted to the proof of:
\begin{proposition}\label{resacy}
The matching $\mathcal{D}_{n,n-2}$ in Corollary~\ref{estesieselbueno} is acyclic. 
\end{proposition}

We have noted that $\mathcal{P}_{12}$ is an acyclic matching in $X_{12}$ so, by Propositions~\ref{descomposicion} and~\ref{x12iscomplex}, the proof of Proposition~\ref{resacy} will follow once we check acyclicity of $\mathcal{Q}$ in $R_{12}$. The acyclicity of $Q_2^3$ follows directly from item 2 in Proposition~\ref{lexi2}, whereas the acyclicity of $Q_1^2$ is proved below by an inductive argument based on Proposition~\ref{relativized} and the following preliminary considerations.

\begin{remark}\label{alrestringir}{\em
The rule $\sigma\mapsto\sigma+(1)$, where the added vertex has label $n$, sets a simplicial complex inclusion $D_{n-1,n-3}\hookrightarrow D_{n,n-2}$. In order to distinguish the referent complex, we will use a superindex ``$m$'' for objects defined within the context of $D_{m,m-2}$. For instance, by definition (alternatively, in view of Proposition~\ref{x12iscomplex}), the intersection of $D_{n-1,n-3}$ with the subcomplex $X_{12}^{n}$ of $D_{n,n-2}$ is the corresponding subcomplex $X_{12}^{n-1}$ of $D_{n-1,n-3}$. In particular the family $R_{12}^{n-1}$ is the intersection of $D_{n-1,n-3}$ and the family $R_{12}^n$. Likewise, $\left(\mathcal{P}_{12}\right)^{n-1}$ is the restriction to $X_{12}^{n-1}$ of $\left(\mathcal{P}_{12}\right)^n$, and from the explicit form of the $\mathcal{Q}_1^2$ matching (see~(\ref{qmatch1}) and~(\ref{qmatch2})), we see that $\left(\mathcal{Q}_1^2\right)^{n-1}$ is the restriction of $\left(\mathcal{Q}_1^2\right)^n$. Thus, in view of Proposition~\ref{relativized}, the goal of this section reduces to proving  Propositions~\ref{partialnocycles} and~\ref{inicioinductivo} below.
}\end{remark}

\begin{proposition}\label{partialnocycles}
For $n\geq5$, there are no $\left(\mathcal{Q}_1^2\right)^n$-cycles in $R^n_{12}-D_{n-1,n-3}$.
\end{proposition}

\begin{proposition}\label{inicioinductivo}
$\left(\mathcal{Q}_1^2\right)^4$ is acyclic.
\end{proposition}

\begin{proof}[Proof of Proposition~\ref{partialnocycles}]
Recall from Section~\ref{secndmt} the notation~(\ref{path}) for directed paths in the modified Hasse diagram coming from a matching. We make a thorough analysis of the possibilities for a directed $\left(\mathcal{Q}_1^2\right)^n$-path in $R_{12}^n-D_{n-1,n-3}$, showing that none of those paths can be a cycle and, as a byproduct, that all of them have in fact length at most three.

In detail, let $\alpha$ be a $\left(\mathcal{Q}_1^2\right)^n$-path in $R_{12}^n-D_{n-1,n-3}$ starting at node $\tau_0$. Since $\tau_0$ lies outside $D_{n-1,n-3}$, the vertex $n$ cannot be an isolated vertex of $\tau_0$. Therefore, the only possibilities for $\tau_0$ are shown below where, for simplicity, we omit the part ``${}+(m)$'' in the graph notation and, as discussed in the proof of Proposition~\ref{lexi}, we are assuming without loss of generality $3\le a<c$ ($a<b$ and $a<c<d$) in the first four (last three) instances:
\begin{equation*}
\begin{tikzpicture}[x=.6cm,y=.6cm]
\node at (0,1.5) {(\emph{i})};
\draw(2,2)--(2,1);\draw(1,2)--(2,2);
\node [above] at (2.45,2) {\scriptsize $b=1$};\node [above] at (1,2) {\scriptsize $a$};
\node [below] at (2.43,1) {\scriptsize $c=n$};
\node at (1,2) {\scriptsize$\bullet$};
\node at (2,2) {\scriptsize$\bullet$};\node at (2,1) {\scriptsize$\bullet$};

\node at (5,1.5) {(\emph{ii})};
\draw(7,2)--(7,1);\draw(6,2)--(7,2);
\node [above] at (7.46,2) {\scriptsize $b=2$};\node [above] at (6,2) {\scriptsize $a$};
\node [below] at (7.44,1) {\scriptsize $c=n$};
\node at (6,2) {\scriptsize$\bullet$};
\node at (7,2) {\scriptsize$\bullet$};\node at (7,1) {\scriptsize$\bullet$};

\node at (10,1.5) {(\emph{iii})};
\draw(12,2)--(12,1);\draw(11,2)--(12,2);
\node [above] at (12.46,2) {\scriptsize $b\ge3$};\node [above] at (11,2) {\scriptsize $a$};
\node [below] at (12.44,1) {\scriptsize $c=n$};
\node at (11,2) {\scriptsize$\bullet$};
\node at (12,2) {\scriptsize$\bullet$};\node at (12,1) {\scriptsize$\bullet$};

\node at (15,1.5) {(\emph{iv})};
\draw(17,2)--(17,1);\draw(16,2)--(17,2);
\node [above] at (17.47,2) {\scriptsize $b=n$};\node [above] at (16,2) {\scriptsize $a$};
\node [below] at (17,1) {\scriptsize $c$};
\node at (16,2) {\scriptsize$\bullet$};
\node at (17,2) {\scriptsize$\bullet$};\node at (17,1) {\scriptsize$\bullet$};

\node at (0,-1.5) {(\emph{v})};
\draw(1,-1)--(2,-1);\draw(1,-2)--(2,-2);
\node [above] at (2.47,-1) {\scriptsize $b=n$};\node [above] at (.5,-1) {\scriptsize $1=a$};
\node [below] at (2,-2) {\scriptsize $d$};\node [below] at (1,-2) {\scriptsize $c$};
\node at (1,-1) {\scriptsize$\bullet$};\node at (1,-2) {\scriptsize$\bullet$};
\node at (2,-1) {\scriptsize$\bullet$};\node at (2,-2) {\scriptsize$\bullet$};

\node at (5,-1.5) {(\emph{vi})};
\draw(6,-1)--(7,-1);\draw(6,-2)--(7,-2);
\node [above] at (7.47,-1) {\scriptsize $b=n$};\node [above] at (5.5,-1) {\scriptsize $1<a$};
\node [below] at (7,-2) {\scriptsize $d$};\node [below] at (6,-2) {\scriptsize $c$};
\node at (6,-1) {\scriptsize$\bullet$};\node at (6,-2) {\scriptsize$\bullet$};
\node at (7,-1) {\scriptsize$\bullet$};\node at (7,-2) {\scriptsize$\bullet$};

\node at (10,-1.5) {(\emph{vii})};
\draw(11,-1)--(12,-1);\draw(12,-2)--(11,-2);
\node [above] at (12,-1) {\scriptsize $b$};\node [above] at (11,-1) {\scriptsize $a$};
\node [below] at (12.47,-2) {\scriptsize $d=n$};\node [below] at (11,-2) {\scriptsize $c$};
\node at (11,-1) {\scriptsize$\bullet$};\node at (11,-2) {\scriptsize$\bullet$};
\node at (12,-1) {\scriptsize$\bullet$};\node at (12,-2) {\scriptsize$\bullet$};
\end{tikzpicture}
\end{equation*}

We consider in full detail the case\footnote{This is a simple but representative instance that will help the reader follow the argument. For the rest of the instances we will simply describe the forced path, from which the reader will have no trouble seeing the corresponding proof details.} where $\tau_0$ has the form in~(\emph{i}): By~(\ref{qmatch2}), $\alpha$ starts with the matching edge
$$\begin{tikzpicture}[x=.6cm,y=.6cm]
\draw(2,2)--(2,1);\draw(1,2)--(2,2);
\node [above] at (2,2) {\scriptsize $1$};\node [above] at (1,2) {\scriptsize $a$};
\node [below] at (2,1) {\scriptsize $n$};
\node at (1,2) {\scriptsize$\bullet$};
\node at (2,2) {\scriptsize$\bullet$};\node at (2,1) {\scriptsize$\bullet$};
\node at (3.5,1.5) {$\nearrow$};

\draw(6,2)--(6,1);\draw(5,2)--(6,2);\draw(5,2)--(5,1);
\node [above] at (6,2) {\scriptsize $1$};\node [above] at (5,2) {\scriptsize $a$};
\node [below] at (6,1) {\scriptsize $n$};\node [below] at (5,1) {\scriptsize $2$};
\node at (5,2) {\scriptsize$\bullet$};\node at (5,1) {\scriptsize$\bullet$};
\node at (6,2) {\scriptsize$\bullet$};\node at (6,1) {\scriptsize$\bullet$};
\end{tikzpicture}$$
After this initial ``step'', there are two options for $\alpha$: either taking the face that deletes the edge $1a$ or the face that deletes the edge $1n$. However, in the latter instance, we arrive at 
$$\begin{tikzpicture}[x=.6cm,y=.6cm]
\draw(5,2)--(6,2);\draw(5,2)--(5,1);
\node [above] at (6,2) {\scriptsize $1$};\node [above] at (5,2) {\scriptsize $a$};
\node [below] at (6,1) {\scriptsize $n$};\node [below] at (5,1) {\scriptsize $2$};
\node at (5,2) {\scriptsize$\bullet$};\node at (5,1) {\scriptsize$\bullet$};
\node at (6,2) {\scriptsize$\bullet$};\node at (6,1) {\scriptsize$\bullet$};
\end{tikzpicture}$$
which is a simplex of $D_{n-1,n-3}$ (as $n$ is isolated). This instance is forbidden since $\alpha$ is a path in $R^n_{12}-D_{n-1,n-3}$. Therefore $\alpha$ is forced to continue as
$$\begin{tikzpicture}[x=.6cm,y=.6cm]
\draw(2,2)--(2,1);\draw(1,2)--(2,2);
\node [above] at (2,2) {\scriptsize $1$};\node [above] at (1,2) {\scriptsize $a$};
\node [below] at (2,1) {\scriptsize $n$};
\node at (1,2) {\scriptsize$\bullet$};
\node at (2,2) {\scriptsize$\bullet$};\node at (2,1) {\scriptsize$\bullet$};

\node at (3.5,1.5) {$\nearrow$};

\draw(6,2)--(6,1);\draw(5,2)--(6,2);\draw(5,2)--(5,1);
\node [above] at (6,2) {\scriptsize $1$};\node [above] at (5,2) {\scriptsize $a$};
\node [below] at (6,1) {\scriptsize $n$};\node [below] at (5,1) {\scriptsize $2$};
\node at (5,2) {\scriptsize$\bullet$};\node at (5,1) {\scriptsize$\bullet$};
\node at (6,2) {\scriptsize$\bullet$};\node at (6,1) {\scriptsize$\bullet$};

\node at (7.5,1.5) {$\searrow$};

\draw(10,2)--(10,1);\draw(9,2)--(9,1);
\node [above] at (10,2) {\scriptsize $1$};\node [above] at (9,2) {\scriptsize $a$};
\node [below] at (10,1) {\scriptsize $n$};\node [below] at (9,1) {\scriptsize $2$};
\node at (9,2) {\scriptsize$\bullet$};\node at (9,1) {\scriptsize$\bullet$};
\node at (10,2) {\scriptsize$\bullet$};\node at (10,1) {\scriptsize$\bullet$};
\end{tikzpicture}$$
According to~(\ref{qmatch1}), the next matching edge in $\alpha$ is 
$$\begin{tikzpicture}[x=.6cm,y=.6cm]
\draw(6,2)--(6,1);\draw(5,2)--(5,1);
\node [above] at (6,2) {\scriptsize $1$};\node [above] at (5,2) {\scriptsize $a$};
\node [below] at (6,1) {\scriptsize $n$};\node [below] at (5,1) {\scriptsize $2$};
\node at (5,2) {\scriptsize$\bullet$};\node at (5,1) {\scriptsize$\bullet$};
\node at (6,2) {\scriptsize$\bullet$};\node at (6,1) {\scriptsize$\bullet$};

\node at (7.5,1.5) {$\nearrow$};

\draw(10,2)--(10,1);\draw(9,2)--(9,1);\draw(9,1)--(10,2);
\node [above] at (10,2) {\scriptsize $1$};\node [above] at (9,2) {\scriptsize $a$};
\node [below] at (10,1) {\scriptsize $n$};\node [below] at (9,1) {\scriptsize $2$};
\node at (9,2) {\scriptsize$\bullet$};\node at (9,1) {\scriptsize$\bullet$};
\node at (10,2) {\scriptsize$\bullet$};\node at (10,1) {\scriptsize$\bullet$};
\end{tikzpicture}$$
However, the latter 2-dimensional simplex lies outside $R_{12}^n$ due to the $12$-edge (recall Remark~\ref{not12}). Putting everything together, we have shown that, in case~(\emph{i}) above, the directed path $\alpha$ is forced to be
\begin{equation}\label{enresumen}
\begin{gathered}
\begin{tikzpicture}[x=.6cm,y=.6cm]
\draw(2,2)--(2,1);\draw(1,2)--(2,2);
\node [above] at (2,2) {\scriptsize $1$};\node [above] at (1,2) {\scriptsize $a$};
\node [below] at (2,1) {\scriptsize $n$};
\node at (1,2) {\scriptsize$\bullet$};
\node at (2,2) {\scriptsize$\bullet$};\node at (2,1) {\scriptsize$\bullet$};

\node at (3.5,1.5) {$\nearrow$};

\draw(6,2)--(6,1);\draw(5,2)--(6,2);\draw(5,2)--(5,1);
\node [above] at (6,2) {\scriptsize $1$};\node [above] at (5,2) {\scriptsize $a$};
\node [below] at (6,1) {\scriptsize $n$};\node [below] at (5,1) {\scriptsize $2$};
\node at (5,2) {\scriptsize$\bullet$};\node at (5,1) {\scriptsize$\bullet$};
\node at (6,2) {\scriptsize$\bullet$};\node at (6,1) {\scriptsize$\bullet$};

\node at (7.5,1.5) {$\searrow$};

\draw(10,2)--(10,1);\draw(9,2)--(9,1);\draw[dotted](9,1)--(10,2);
\node [above] at (10,2) {\scriptsize $1$};\node [above] at (9,2) {\scriptsize $a$};
\node [below] at (10,1) {\scriptsize $n$};\node [below] at (9,1) {\scriptsize $2$};
\node at (9,2) {\scriptsize$\bullet$};\node at (9,1) {\scriptsize$\bullet$};
\node at (10,2) {\scriptsize$\bullet$};\node at (10,1) {\scriptsize$\bullet$};
\end{tikzpicture}
\end{gathered}
\end{equation}
where the dotted edge indicates the reason why $\alpha$ stops (before leaving $R^n_{12}$). As promised, this is a non-cycle of length at most 3 (actually length 1).

The above analysis goes through in all other six cases. Namely, in each case the path $\alpha$ in $R^n_{12}-D_{n-1,n-3}$ is forced to be the one indicated below. The only exception is reflected by the two items with label ``(\emph{iv})'' below, due to the two corresponding options for the path $\alpha$. As above, all of the paths stop at the indicated node just before leaving $R_{12}^n$ due to the matching edge suggested by the dotted $12$-edge (just as in~(\ref{enresumen}) above). As in the detailed explanation above, the reader should keep in mind that, since $\alpha$ is a path in $R_{12}^n-D_{n-1,n-3}$, the vertex with label $n$ \emph{must} be part of an edge at all nodes of~$\alpha$ ---thus forcing the indicated behavior of the path. Also worth remarking is the fact that equality symbol in item (\emph{vi}) below is meant to highlight that the condition $2<a$ is forced, for otherwise the node would lie outside $R_{12}^n$. A similar situation holds in item (\emph{vii}) below.

\medskip\noindent
$\begin{tikzpicture}[x=.6cm,y=.6cm]
\node at (-.5,1.5) {(\emph{ii})};
\draw(2,2)--(2,1);\draw(1,2)--(2,2);
\node [above] at (2,2) {\scriptsize $2$};\node [above] at (1,2) {\scriptsize $a$};
\node [below] at (2,1) {\scriptsize $n$};
\node at (1,2) {\scriptsize$\bullet$};
\node at (2,2) {\scriptsize$\bullet$};\node at (2,1) {\scriptsize$\bullet$};

\node at (3.5,1.5) {$\nearrow$};

\draw(6,2)--(6,1);\draw(5,2)--(6,2);\draw(5,2)--(5,1);
\node [above] at (6,2) {\scriptsize $2$};\node [above] at (5,2) {\scriptsize $a$};
\node [below] at (6,1) {\scriptsize $n$};\node [below] at (5,1) {\scriptsize $1$};
\node at (5,2) {\scriptsize$\bullet$};\node at (5,1) {\scriptsize$\bullet$};
\node at (6,2) {\scriptsize$\bullet$};\node at (6,1) {\scriptsize$\bullet$};

\node at (7.5,1.5) {$\searrow$};

\draw(10,2)--(10,1);\draw(9,2)--(9,1);\draw[dotted](9,1)--(10,2);
\node [above] at (10,2) {\scriptsize $2$};\node [above] at (9,2) {\scriptsize $a$};
\node [below] at (10,1) {\scriptsize $n$};\node [below] at (9,1) {\scriptsize $1$};
\node at (9,2) {\scriptsize$\bullet$};\node at (9,1) {\scriptsize$\bullet$};
\node at (10,2) {\scriptsize$\bullet$};\node at (10,1) {\scriptsize$\bullet$};
\end{tikzpicture}$

\smallskip\noindent 
$\begin{tikzpicture}[x=.6cm,y=.6cm]
\node at (-.5,1.5) {(\emph{iii})};
\draw(2,2)--(2,1);\draw(1,2)--(2,2);
\node [above] at (2.47,2) {\scriptsize $b\ge3$};\node [above] at (1,2) {\scriptsize $a$};
\node [below] at (2,1) {\scriptsize $n$};
\node at (1,2) {\scriptsize$\bullet$};
\node at (2,2) {\scriptsize$\bullet$};\node at (2,1) {\scriptsize$\bullet$};

\node at (3.5,1.5) {$\nearrow$};

\draw(6,2)--(6,1);\draw(5,2)--(6,2);\draw(5,2)--(5,1);
\node [above] at (6.47,2) {\scriptsize $b\ge3$};\node [above] at (5,2) {\scriptsize $a$};
\node [below] at (6,1) {\scriptsize $n$};\node [below] at (5,1) {\scriptsize $1$};
\node at (5,2) {\scriptsize$\bullet$};\node at (5,1) {\scriptsize$\bullet$};
\node at (6,2) {\scriptsize$\bullet$};\node at (6,1) {\scriptsize$\bullet$};

\node at (7.5,1.5) {$\searrow$};

\draw(10,2)--(10,1);\draw(9,2)--(9,1);
\node [above] at (10.47,2) {\scriptsize $b\geq3$};\node [above] at (9,2) {\scriptsize $a$};
\node [below] at (10,1) {\scriptsize $n$};\node [below] at (9,1) {\scriptsize $1$};
\node at (9,2) {\scriptsize$\bullet$};\node at (9,1) {\scriptsize$\bullet$};
\node at (10,2) {\scriptsize$\bullet$};\node at (10,1) {\scriptsize$\bullet$};

\node at (11.5,1.5) {$\nearrow$};

\draw(14,2)--(14,1);\draw(13,2)--(13,1);\draw(13,1)--(14,2);
\node [above] at (14.47,2) {\scriptsize $b\geq3$};\node [above] at (13,2) {\scriptsize $a$};
\node [below] at (14,1) {\scriptsize $n$};\node [below] at (13,1) {\scriptsize $1$};
\node at (13,2) {\scriptsize$\bullet$};\node at (13,1) {\scriptsize$\bullet$};
\node at (14,2) {\scriptsize$\bullet$};\node at (14,1) {\scriptsize$\bullet$};

\node at (15.5,1.5) {$\searrow$};

\draw(18,2)--(18,1);\draw[dotted](17,2)--(17,1);\draw(17,1)--(18,2);
\node [above] at (18.47,2) {\scriptsize $b\geq3$};\node [above] at (17,2) {\scriptsize $2$};
\node [below] at (18,1) {\scriptsize $n$};\node [below] at (17,1) {\scriptsize $1$};
\node at (17,2) {\scriptsize$\circ$};\node at (17,1) {\scriptsize$\bullet$};
\node at (18,2) {\scriptsize$\bullet$};\node at (18,1) {\scriptsize$\bullet$};
\end{tikzpicture}$

\smallskip\noindent 
$\begin{tikzpicture}[x=.6cm,y=.6cm]
\node at (-.5,1.5) {(\emph{iv})};
\draw(2,2)--(2,1);\draw(1,2)--(2,2);
\node [above] at (2,2) {\scriptsize $n$};\node [above] at (1,2) {\scriptsize $a$};
\node [below] at (2,1) {\scriptsize $c$};
\node at (1,2) {\scriptsize$\bullet$};
\node at (2,2) {\scriptsize$\bullet$};\node at (2,1) {\scriptsize$\bullet$};

\node at (3.5,1.5) {$\nearrow$};

\draw(6,2)--(6,1);\draw(5,2)--(6,2);\draw(5,2)--(5,1);
\node [above] at (6,2) {\scriptsize $n$};\node [above] at (5,2) {\scriptsize $a$};
\node [below] at (6,1) {\scriptsize $c$};\node [below] at (5,1) {\scriptsize $1$};
\node at (5,2) {\scriptsize$\bullet$};\node at (5,1) {\scriptsize$\bullet$};
\node at (6,2) {\scriptsize$\bullet$};\node at (6,1) {\scriptsize$\bullet$};

\node at (7.5,1.5) {$\searrow$};

\draw(9,2)--(10,2);\draw(9,2)--(9,1);\draw[dotted](9,1)--(10,1);
\node [above] at (10,2) {\scriptsize $n$};\node [above] at (9,2) {\scriptsize $a$};
\node [below] at (10,1) {\scriptsize $2$};\node [below] at (9,1) {\scriptsize $1$};
\node at (9,2) {\scriptsize$\bullet$};\node at (9,1) {\scriptsize$\bullet$};
\node at (10,2) {\scriptsize$\bullet$};\node at (10,1) {\scriptsize$\circ$};
\end{tikzpicture}$

\smallskip\noindent 
$\begin{tikzpicture}[x=.6cm,y=.6cm]
\node at (-.5,1.5) {(\emph{iv})};
\draw(2,2)--(2,1);\draw(1,2)--(2,2);
\node [above] at (2,2) {\scriptsize $n$};\node [above] at (1,2) {\scriptsize $a$};
\node [below] at (2,1) {\scriptsize $c$};
\node at (1,2) {\scriptsize$\bullet$};
\node at (2,2) {\scriptsize$\bullet$};\node at (2,1) {\scriptsize$\bullet$};

\node at (3.5,1.5) {$\nearrow$};

\draw(6,2)--(6,1);\draw(5,2)--(6,2);\draw(5,2)--(5,1);
\node [above] at (6,2) {\scriptsize $n$};\node [above] at (5,2) {\scriptsize $a$};
\node [below] at (6,1) {\scriptsize $c$};\node [below] at (5,1) {\scriptsize $1$};
\node at (5,2) {\scriptsize$\bullet$};\node at (5,1) {\scriptsize$\bullet$};
\node at (6,2) {\scriptsize$\bullet$};\node at (6,1) {\scriptsize$\bullet$};

\node at (7.5,1.5) {$\searrow$};

\draw(10,1)--(10,2);\draw(9,2)--(9,1);
\node [above] at (10,2) {\scriptsize $n$};\node [above] at (9,2) {\scriptsize $a$};
\node [below] at (10,1) {\scriptsize $c$};\node [below] at (9,1) {\scriptsize $1$};
\node at (9,2) {\scriptsize$\bullet$};\node at (9,1) {\scriptsize$\bullet$};
\node at (10,2) {\scriptsize$\bullet$};\node at (10,1) {\scriptsize$\bullet$};

\node at (11.5,1.5) {$\nearrow$};

\draw(14,1)--(14,2);\draw(13,2)--(13,1);\draw(13,1)--(14,1);
\node [above] at (14,2) {\scriptsize $n$};\node [above] at (13,2) {\scriptsize $a$};
\node [below] at (14,1) {\scriptsize $c$};\node [below] at (13,1) {\scriptsize $1$};
\node at (13,2) {\scriptsize$\bullet$};\node at (13,1) {\scriptsize$\bullet$};
\node at (14,2) {\scriptsize$\bullet$};\node at (14,1) {\scriptsize$\bullet$};

\node at (15.5,1.5) {$\searrow$};

\draw(18,1)--(18,2);\draw(17,1)--(18,1);\draw[dotted](17,1)--(17,2);
\node [above] at (18,2) {\scriptsize $n$};\node [above] at (17,2) {\scriptsize $2$};
\node [below] at (18,1) {\scriptsize $c$};\node [below] at (17,1) {\scriptsize $1$};
\node at (17,2) {\scriptsize$\circ$};\node at (17,1) {\scriptsize$\bullet$};
\node at (18,2) {\scriptsize$\bullet$};\node at (18,1) {\scriptsize$\bullet$};
\end{tikzpicture}$

\smallskip\noindent 
$\begin{tikzpicture}[x=.6cm,y=.6cm]
\node at (-.5,1.5) {(\emph{v})};
\draw(1,1)--(2,1);\draw(1,2)--(2,2);
\node [above] at (2,2) {\scriptsize $n$};\node [above] at (1,2) {\scriptsize $1$};
\node [below] at (2,1) {\scriptsize $d$};\node [below] at (1,1) {\scriptsize $c$};
\node at (1,2) {\scriptsize$\bullet$};\node at (1,1) {\scriptsize$\bullet$};
\node at (2,2) {\scriptsize$\bullet$};\node at (2,1) {\scriptsize$\bullet$};

\node at (3.5,1.5) {$\nearrow$};

\draw(5,1)--(6,1);\draw(5,2)--(6,2);\draw(5,1)--(5,2);
\node [above] at (6,2) {\scriptsize $n$};\node [above] at (5,2) {\scriptsize $1$};
\node [below] at (6,1) {\scriptsize $d$};\node [below] at (5,1) {\scriptsize $c$};
\node at (5,2) {\scriptsize$\bullet$};\node at (5,1) {\scriptsize$\bullet$};
\node at (6,2) {\scriptsize$\bullet$};\node at (6,1) {\scriptsize$\bullet$};

\node at (7.5,1.5) {$\searrow$};

\draw(9,2)--(10,2);\draw(9,2)--(9,1);
\node [above] at (10,2) {\scriptsize $n$};\node [above] at (9,2) {\scriptsize $1$};
\node [below] at (9,1) {\scriptsize $c$};
\node at (9,2) {\scriptsize$\bullet$};\node at (9,1) {\scriptsize$\bullet$};
\node at (10,2) {\scriptsize$\bullet$};

\node at (11.5,1.5) {$\nearrow$};

\draw(13,2)--(14,2);\draw(13,2)--(13,1);\draw(13,1)--(14,1);
\node [above] at (14,2) {\scriptsize $n$};\node [above] at (13,2) {\scriptsize $1$};
\node [below] at (14,1) {\scriptsize $2$};\node [below] at (13,1) {\scriptsize $c$};
\node at (13,2) {\scriptsize$\bullet$};\node at (13,1) {\scriptsize$\bullet$};
\node at (14,2) {\scriptsize$\bullet$};\node at (14,1) {\scriptsize$\bullet$};

\node at (15.5,1.5) {$\searrow$};

\draw(17,2)--(18,2);\draw(17,1)--(18,1);\draw[dotted](18,1)--(17,2);
\node [above] at (18,2) {\scriptsize $n$};\node [above] at (17,2) {\scriptsize $1$};
\node [below] at (18,1) {\scriptsize $2$};\node [below] at (17,1) {\scriptsize $c$};
\node at (17,2) {\scriptsize$\bullet$};\node at (17,1) {\scriptsize$\bullet$};
\node at (18,2) {\scriptsize$\bullet$};\node at (18,1) {\scriptsize$\bullet$};
\end{tikzpicture}$

\smallskip\noindent 
$\begin{tikzpicture}[x=.6cm,y=.6cm]
\node at (-.5,1.5) {(\emph{vi})};
\draw(1,1)--(2,1);\draw(1,2)--(2,2);
\node [above] at (2,2) {\scriptsize $n$};\node [above] at (.6,2) {\scriptsize $1<a$};
\node [below] at (2,1) {\scriptsize $d$};\node [below] at (1,1) {\scriptsize $c$};
\node at (1,2) {\scriptsize$\bullet$};\node at (1,1) {\scriptsize$\bullet$};
\node at (2,2) {\scriptsize$\bullet$};\node at (2,1) {\scriptsize$\bullet$};

\node at (3.5,1.5) {$\nearrow$};

\draw(5,1)--(6,1);\draw(5,2)--(6,2);\draw(5,2)--(5,1);
\node [above] at (6,2) {\scriptsize $n$};\node [above] at (4.6,2) {\scriptsize $1<a$};
\node [below] at (6,1) {\scriptsize $d$};\node [below] at (5,1) {\scriptsize $c$};
\node at (5,2) {\scriptsize$\bullet$};\node at (5,1) {\scriptsize$\bullet$};
\node at (6,2) {\scriptsize$\bullet$};\node at (6,1) {\scriptsize$\bullet$};

\node at (7.5,1.5) {$\searrow$};

\draw(9,2)--(10,2);\draw(9,2)--(9,1);
\node [above] at (10,2) {\scriptsize $n$};\node [above] at (8.6,2) {\scriptsize $1<a$};
\node [below] at (9,1) {\scriptsize $c$};
\node at (9,2) {\scriptsize$\bullet$};\node at (9,1) {\scriptsize$\bullet$};
\node at (10,2) {\scriptsize$\bullet$};

\node at (11.5,1.5) {$\nearrow$};

\draw(13,2)--(14,2);\draw(13,2)--(13,1);\draw(13,1)--(14,1);
\node [above] at (14,2) {\scriptsize $n$};\node [above] at (12.6,2) {\scriptsize $1<a$};
\node [below] at (14,1) {\scriptsize $1$};\node [below] at (13,1) {\scriptsize $c$};
\node at (13,2) {\scriptsize$\bullet$};\node at (13,1) {\scriptsize$\bullet$};
\node at (14,2) {\scriptsize$\bullet$};\node at (14,1) {\scriptsize$\bullet$};

\node at (15.5,1.5) {$\searrow$};

\draw(17,2)--(18,2);\draw(17,1)--(18,1);
\node [above] at (18,2) {\scriptsize $n$};\node [above] at (16.6,2) {\scriptsize $1<a$};
\node [below] at (18,1) {\scriptsize $1$};\node [below] at (17,1) {\scriptsize $c$};
\node at (17,2) {\scriptsize$\bullet$};\node at (17,1) {\scriptsize$\bullet$};
\node at (18,2) {\scriptsize$\bullet$};\node at (18,1) {\scriptsize$\bullet$};

\node at (19.5,1.5) {$=$};

\draw(21,2)--(22,2);\draw(21,1)--(22,1);
\node [above] at (22,2) {\scriptsize $n$};\node [above] at (20.6,2) {\scriptsize $2<a$};
\node [below] at (22,1) {\scriptsize $1$};\node [below] at (21,1) {\scriptsize $c$};
\node at (21,2) {\scriptsize$\bullet$};\node at (21,1) {\scriptsize$\bullet$};
\node at (22,2) {\scriptsize$\bullet$};\node at (22,1) {\scriptsize$\bullet$};
\end{tikzpicture}$

\smallskip\noindent 
$\begin{tikzpicture}[x=.6cm,y=.6cm]
\node at (15,1) {\rule{0mm}{2mm}};
\node at (19.5,1.5) {$\nearrow$};

\draw(21,2)--(22,2);\draw(21,1)--(22,1);\draw(22,1)--(21,2);
\node [above] at (22,2) {\scriptsize $n$};\node [above] at (20.6,2) {\scriptsize $2<a$};
\node [below] at (22,1) {\scriptsize $1$};\node [below] at (21,1) {\scriptsize $c$};
\node at (21,2) {\scriptsize$\bullet$};\node at (21,1) {\scriptsize$\bullet$};
\node at (22,2) {\scriptsize$\bullet$};\node at (22,1) {\scriptsize$\bullet$};

\node at (23.5,1.5) {$\searrow$};

\draw(25,2)--(26,2);\draw(26,1)--(25,2);\draw[dotted](25,1)--(26,1);
\node [above] at (26,2) {\scriptsize $n$};\node [above] at (24.6,2) {\scriptsize $2<a$};
\node [below] at (26,1) {\scriptsize $1$};\node [below] at (25,1) {\scriptsize $2$};
\node at (26,2) {\scriptsize$\bullet$};\node at (25,1) {\scriptsize$\circ$};
\node at (25,2) {\scriptsize$\bullet$};\node at (26,1) {\scriptsize$\bullet$};
\end{tikzpicture}$

\smallskip\noindent 
$\begin{tikzpicture}[x=.6cm,y=.6cm]
\node at (-.5,1.5) {(\emph{vii})};
\draw(1,1)--(2,1);\draw(1,2)--(2,2);
\node [above] at (2,2) {\scriptsize $b$};\node [above] at (1,2) {\scriptsize $a$};
\node [below] at (2,1) {\scriptsize $n$};\node [below] at (1,1) {\scriptsize $c$};
\node at (1,2) {\scriptsize$\bullet$};\node at (1,1) {\scriptsize$\bullet$};
\node at (2,2) {\scriptsize$\bullet$};\node at (2,1) {\scriptsize$\bullet$};

\node at (3.5,1.5) {$\nearrow$};

\draw(5,2)--(5,1);\draw(5,2)--(6,2);\draw(6,1)--(5,1);
\node [above] at (6,2) {\scriptsize $b$};\node [above] at (5,2) {\scriptsize $a$};
\node [below] at (6,1) {\scriptsize $n$};\node [below] at (5,1) {\scriptsize $c$};
\node at (5,2) {\scriptsize$\bullet$};\node at (5,1) {\scriptsize$\bullet$};
\node at (6,2) {\scriptsize$\bullet$};\node at (6,1) {\scriptsize$\bullet$};

\node at (7.5,1.5) {$\searrow$};

\draw(10,1)--(9,1);\draw(9,2)--(9,1);
\node [above] at (9,2) {\scriptsize $a$};
\node [below] at (10,1) {\scriptsize $n$};\node [below] at (9,1) {\scriptsize $c$};
\node at (9,2) {\scriptsize$\bullet$};\node at (9,1) {\scriptsize$\bullet$};
\node at (10,1) {\scriptsize$\bullet$};

\node at (11.5,1.5) {$=$};

\draw(14,1)--(13,1);\draw(13,2)--(13,1);
\node [above] at (12.6,2) {\scriptsize $2<a$};
\node [below] at (14,1) {\scriptsize $n$};\node [below] at (13,1) {\scriptsize $c$};
\node at (13,2) {\scriptsize$\bullet$};\node at (13,1) {\scriptsize$\bullet$};
\node at (14,1) {\scriptsize$\bullet$};

\node at (15.5,1.5) {$\nearrow$};

\draw(17,2)--(18,2);\draw(17,1)--(18,1);\draw(17,1)--(17,2);
\node [above] at (18,2) {\scriptsize $1$};\node [above] at (16.6,2) {\scriptsize $2<a$};
\node [below] at (18,1) {\scriptsize $n$};\node [below] at (17,1) {\scriptsize $c$};
\node at (17,2) {\scriptsize$\bullet$};\node at (17,1) {\scriptsize$\bullet$};
\node at (18,2) {\scriptsize$\bullet$};\node at (18,1) {\scriptsize$\bullet$};

\node at (19.5,1.5) {$\searrow$};

\draw(21,2)--(22,2);\draw(21,1)--(22,1);
\node [above] at (22,2) {\scriptsize $1$};\node [above] at (20.6,2) {\scriptsize $2<a$};
\node [below] at (22,1) {\scriptsize $n$};\node [below] at (21,1) {\scriptsize $c$};
\node at (21,2) {\scriptsize$\bullet$};\node at (21,1) {\scriptsize$\bullet$};
\node at (22,2) {\scriptsize$\bullet$};\node at (22,1) {\scriptsize$\bullet$};
\end{tikzpicture}$

\smallskip\noindent 
$\begin{tikzpicture}[x=.6cm,y=.6cm]
\node at (15,1) {\rule{0mm}{2mm}};
\node at (19.5,1.5) {$\nearrow$};

\draw(21,2)--(22,2);\draw(21,1)--(22,1);\draw(22,2)--(21,1);
\node [above] at (22,2) {\scriptsize $1$};\node [above] at (20.6,2) {\scriptsize $2<a$};
\node [below] at (22,1) {\scriptsize $n$};\node [below] at (21,1) {\scriptsize $c$};
\node at (21,2) {\scriptsize$\bullet$};\node at (21,1) {\scriptsize$\bullet$};
\node at (22,2) {\scriptsize$\bullet$};\node at (22,1) {\scriptsize$\bullet$};

\node at (23.5,1.5) {$\searrow$};

\draw(25,1)--(26,1);\draw(26,2)--(25,1);\draw[dotted](25,2)--(26,2);
\node [above] at (26,2) {\scriptsize $1$};\node [above] at (25,2) {\scriptsize $2$};
\node [below] at (26,1) {\scriptsize $n$};\node [below] at (25,1) {\scriptsize $c$};
\node at (26,2) {\scriptsize$\bullet$};\node at (25,1) {\scriptsize$\bullet$};
\node at (25,2) {\scriptsize$\circ$};\node at (26,1) {\scriptsize$\bullet$};
\end{tikzpicture}$

\medskip\noindent
What is relevant for us is that all of these paths are non-cycles of length at most 3.
\end{proof}

\begin{proof}[Proof of Proposition~\ref{inicioinductivo}]
This is a strightforward calculation. Recall from Example~\ref{d42} that $D_{4,2}$ is the complex with 2-dimen\-sional facets (i.e., maximal faces) $12|13|23$, $12|14|24$, $13|14|34$ and $23|24|34$, and 3-dimen\-sional facets $12|13|24|34$, $12|14|23|34$ and $13|14|23|24$. Four of these seven facets contain the edge $12$ and, therefore, (together with their faces) lie in $X_{12}$. Thus, $R_{12}$ consists of (some) faces of the facets $13|14|34$, $23|24|34$ and $13|14|23|24$, and such faces have dimensions in between 1 and 3, in view of Remark~\ref{not12}. Explicitly, and by direct inspection (keeping in mind Equation~(\ref{r12}) in Proposition~\ref{summa}), $R_{12}$ consists of the nine simplices indicated in the following table, where rows indicate the facet giving rise to the shown face of $R_{12}$:

\bigskip\smallskip\centerline{\begin{tabular}{|c|c|c|c|}
\hline facet &  $\dim=1$ & $\dim=2$ & $\dim=3$ \\ \hline
 $13|14|34$ &  $13|14$ & $13|14|34$ & \\ \hline
 $23|24|34$ &  $23|24$ & $23|24|34$ &  \\ \hline
 $13|14|23|24$  &  $13|14$, \,$23|24$ & $13|14|23$,\, $13|14|24$,\, $13|23|24$,\, $14|23|24$ & $13|14|23|24$ \\
 \hline
\end{tabular}}

\bigskip\medskip
The acyclicity of $\mathcal{Q}_1^2$ is now evident from its modified Hasse diagram:

\medskip\medskip\noindent 
\raisebox{-2.3mm}{$\begin{tikzpicture}[x=.6cm,y=.6cm]
\node at (-3,0) {\rule{0mm}{4mm}};
\node[above] at (0,0) {$13|14|34$};
\node[above] at (4,0) {$23|24|34$};
\node[above] at (8,0) {$13|14|23$};
\node[above] at (12,0) {$13|14|24$};
\node[above] at (16,0) {$13|23|24$};
\node[above] at (20,0) {$14|23|24$};
\node[below] at (8,-2) {$13|14$};
\node[below] at (12,-2) {$23|24$};
\draw[->,very thick](12,-2)--(16,0);
\draw[->,very thick](8,-2)--(8,0);
\draw[->,very thin](0,0)--(7.6,-2);
\draw[->,very thin](4,0)--(11.6,-2);
\draw[->,very thin](12,0)--(8.4,-2);
\draw[->,very thin](20,0)--(12.4,-2);
\end{tikzpicture}$}
\end{proof}

\section{Proof of the main theorems}\label{graficasinfinitas}
All of the hard work has been accounted for in the previous sections; here we only put the pieces together. 

\begin{proof}[Proof of Theorem~\ref{teoremaprincipal}]
The fact that $D_{n,n-2}$ has the homotopy type of a wedge of 2-dimensional spheres follows from Theorem~\ref{DMTkey}, Corollary~\ref{estesieselbueno} and Proposition~\ref{resacy}. It remains to count the number $N_n$ of $S^2$ wedge summands in the homotopy type of $D_{n,n-2}$. This is computed in terms of the Euler characteristic of $D_{n,n-2}$, namely
$$
N_n+1=\chi(D_{n,n-2})=c_0-c_1+c_2-c_3,
$$
where $c_i$ stands for the number of $i$-dimensional faces in $D_{n,n-2}$. Remark~\ref{partcs} and the characterization of $d$-dimensional simplices of $D_{n,n-2}$ discussed in Section~\ref{structuresection} give 
$$
\mbox{$c_0=\genfrac(){0pt}{1}{n}{2}$, \ $c_1=\genfrac(){0pt}{1}{c_0}{2}$, \ $c_2=\genfrac(){0pt}{1}{n}{3}+12\genfrac(){0pt}{1}{n}{4}$ \ and \ $c_3=3\genfrac(){0pt}{1}{n}{4}$,}
$$
from which the expression for $N_n$ in~(\ref{nsw}) follows after a little arithmetics.
\end{proof}

The proof of Theorem~\ref{teoremaprincipalextendido} is identical, except that the use of Theorem~\ref{DMTkey} has to be replaced by Theorem~\ref{DMTkeyinfinito}. We start by identifying $D^2$ as a graph complex.

\begin{lemma}
$D^2$ is the complex of graphs on an infinite number of vertices, almost all being isolated, and the rest forming a subgraph of either a 3-cycle or a 4-cycle.
\end{lemma}
\begin{proof}
As explained in Remark~\ref{alrestringir}, $D_{n-1,n-3}$ is a subcomplex of $D_{n,n-2}$, namley, we think of (the simplex determined by) a graph $\sigma\in D_{n-1,n-3}$ as (the simplex determined by) the graph $\sigma+(1)\in D_{n,n-2}$, where the added vertex has label $n$. Since $D^2$ is the limit complex of the sequence $D_{3,1}\hookrightarrow D_{4,2}\hookrightarrow\cdots$, the lemma follows from Theorem~\ref{vizing} and Lemma~\ref{combinatorics2}.
\end{proof}

\begin{proof}[Proof of Theorem~\ref{teoremaprincipalextendido}]
Recall from Remark~\ref{alrestringir} that the acyclic matching $\mathcal{D}_{n,n-2}$ in Corollary~\ref{estesieselbueno} restricts (in the sense of item (\emph{iii}) of Proposition~\ref{relativized}) to the matching $\mathcal{D}_{n-1,n-3}$. Consequently the modified Hasse diagram of $\mathcal{D}_{n-1,n-3}$ is contained in that for $\mathcal{D}_{n,n-2}$. The union of all these directed graphs (as $n$ grows) yields the modified Hasse diagram of the union $\mathcal{D}^2=\cup_n\mathcal{D}_{n,n-2}$. Since the directed paths of $\mathcal{D}^2$ are the union of the directed paths on the several $\mathcal{D}_{n,n-2}$, $\mathcal{D}^2$ inherits being acyclic and having no infinite directed simple paths. Likewise, the critical simplices of $\mathcal{D}^2$ are given by the union of the critical simplices of the several $\mathcal{D}_{n,n-2}$. Thus $\mathcal{D}^2$ has a countable number of critical 2-simplexes, together with a single 0-dimensional simplex ---the graph with vertices given by the natural numbers $\{1,2,3,\ldots\}$, and a single edge $12$. Theorem~\ref{DMTkeyinfinito} applies to complete the proof of Theorem~\ref{teoremaprincipalextendido}.
\end{proof}

\section{The case of \texorpdfstring{$D_{n,n-3}$}{dnn3}}\label{52}
\begin{proof}[Proof of Proposition~\ref{d52}]
The argument is by direct computation, as in the proof of Proposition~\ref{inicioinductivo}. We provide a complete roadmap that simplifies the task of verifying details.

Note that $D_{5,2}$ has facets in dimensions 5 and 6; the latter ones are given by Vizing's Theorem~\ref{vizing}, and the former ones are of the form $K_4+(1)$, where $K_4$ stands for a complete graph on four vertices (taken from the set $\{1,2,3,4,5\}$). There are $\genfrac(){0pt}{1}{5}{4}=5$ possibilities for the graph $K_4$, three of which have the edge 12 and, therefore, lie (together with their faces) in $X_{12}$. The other two possibilities for $K_4$ are
\begin{equation}\label{facetas1}
\raisebox{-3.5mm}{
\begin{tikzpicture}[x=.6cm,y=.6cm]
\node[left] at (0,1.5) {\scriptsize$1$};
\draw(2,2)--(2,1);\draw(1,2)--(2,2);\draw(1,1)--(2,2);\draw(1,1)--(1,2);\draw(1,1)--(2,1);\draw(1,2)--(2,1);
\node at (0,1.5) {\scriptsize$\bullet$}; \node at (1,1) {\scriptsize$\bullet$};
\node at (1,2) {\scriptsize$\bullet$};\node at (2,2) {\scriptsize$\bullet$};\node at (2,1) {\scriptsize$\bullet$};
\node at (4.3,1.5) {and};
\node[left] at (7,1.5) {\scriptsize$2$};
\draw(9,2)--(9,1);\draw(8,2)--(9,2);\draw(8,1)--(9,2);\draw(8,1)--(8,2);\draw(8,1)--(9,1);\draw(8,2)--(9,1);
\node at (7,1.5) {\scriptsize$\bullet$}; \node at (8,1) {\scriptsize$\bullet$};
\node at (8,2) {\scriptsize$\bullet$};\node at (9,2) {\scriptsize$\bullet$};\node at (9,1) {\scriptsize$\bullet$};
\node at (9.4,1.1) {,};
\end{tikzpicture}}\end{equation}
some of whose faces belong to $R_{12}$. Explicitly, and by direct inspection  (keeping in mind Equation~(\ref{r12}) in Proposition~\ref{summa}), the simplices in $R_{12}$ that come from these two facets are indicated in Tables~\ref{tabla1} and~\ref{tabla2}.

\medskip
\begin{table}
\centerline{\begin{tabular}{|c|c|c|c|}
\hline $\dim=2$ & $\dim=3$ & $\dim=4$ & $\dim=5$ \\ \hline
$23|24|25$ &  $23|24|25|34$ & $23|24|25|34|35$ & $23|24|25|34|35|45$ \\ 
&  $23|24|25|35$ & $23|24|25|34|45$ & \\ 
& $23|24|25|45$ & $23|24|25|35|45$ & \\ \hline
\end{tabular}}
\caption{Simplices in $R_{12}$ coming from the first facet in~(\ref{facetas1})}
\label{tabla1}
\end{table}

\begin{table}
\centerline{\begin{tabular}{|c|c|c|c|}
\hline $\dim=2$ & $\dim=3$ & $\dim=4$ & $\dim=5$ \\ \hline
$13|14|15$ &  $13|14|15|34$ & $13|14|15|34|35$ & $13|14|15|34|35|45$ \\ 
&  $13|14|15|35$ & $13|14|15|34|45$ & \\ 
& $13|14|15|45$ & $13|14|15|35|45$ & \\ \hline
\end{tabular}}
\caption{Simplices in $R_{12}$ coming from the second facet in~(\ref{facetas1})}
\label{tabla2}
\end{table}

On the other hand, the 6-dimensional facets of $D_{5,2}$ have the form $K'_5$ where, as described in Vizing's Theorem, $K'_5$ is the complement (in the complete graph on vertices $\{1,2,3,4,5\}$) of a graph $K''_5$ of the form
$$
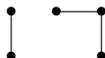
\begin{tikzpicture}[x=.6cm,y=.6cm]
\draw(2,2)--(2,1);\draw(1,2)--(2,2);\draw(0,1)--(0,2);
\node at (0,1){\scriptsize$\bullet$}; \node at (0,2) {\scriptsize$\bullet$};
\node at (1,2) {\scriptsize$\bullet$};\node at (2,2) {\scriptsize$\bullet$};\node at (2,1) {\scriptsize$\bullet$};
\end{tikzpicture}$$
There are $3\genfrac(){0pt}{1}{5}{3}=30$ possibilities for $K''_5$, and 21 of them do not contain the edge 12, so that the corresponding facet in $D_{5,2}$ (which contains the edge $12$) does not contribute to $R_{12}$. The nine 6-dimensional facets of $D_{5,2}$ that contribute to $R_{12}$, together with the corresponding simplices in $R_{12}$, come from the cases where $K''_5$ contains the edge $12$ (in view of~(\ref{r12})), and are indicated in Tables~\ref{tabla3}--\ref{tabla11}. We omit writing an element that should appear in Table~$i$, if the element has been listed in Table~$j$ for some $j<i$.

\begin{table}[h!]
\centerline{\begin{tabular}{|c|c|c|c|}
\hline $\dim=3$ & $\dim=4$ & $\dim=5$ & $\dim=6$ \\ \hline
$14|23|24|25$ &  $14|15|23|24|25$ &  $14|15|23|24|25|34$ & $14|15|23|24|25|34|35$ \\ 
$15|23|24|25$ &  $14|23|24|25|34$ & $14|15|23|24|25|35$ & \\ 
&  $14|23|24|25|35$ & $14|23|24|25|34|35$ & \\ 
&  $15|23|24|25|34$ & $15|23|24|25|34|35$ & \\ 
&  $15|23|24|25|35$ & & \\\hline
\end{tabular}}
\caption{Simplices in $R_{12}$ coming from $K''_5\;=\;\;
\stackrel{2}{\bullet}\hspace{-2mm}\raisebox{.9mm}{\rule{6mm}{.2mm}}\hspace{-2mm}
\stackrel{1}{\bullet}\hspace{-2mm}\raisebox{.9mm}{\rule{6mm}{.2mm}}\hspace{-2mm}
\stackrel{3}{\bullet}\hspace{4mm} 
\stackrel{4}{\bullet}\hspace{-2mm}\raisebox{.9mm}{\rule{6mm}{.2mm}}\hspace{-2mm}
\stackrel{5}{\bullet}\hspace{-2mm}$}
\label{tabla3}
\end{table}

\begin{table}[h!]
\centerline{\begin{tabular}{|c|c|c|c|}
\hline $\dim=3$ & $\dim=4$ & $\dim=5$ & $\dim=6$ \\ \hline
$13|23|24|25$ &  $13|15|23|24|25$ & $13|15|23|24|25|34$ & $13|15|23|24|25|34|45$ \\ 
& $13|23|24|25|34$ & $13|15|23|24|25|45$ & \\ 
& $13|23|24|25|45$ & $13|23|24|25|34|45$ & \\ 
& $15|23|24|25|45$ & $15|23|24|25|34|45$ & \\\hline
\end{tabular}}
\caption{Simplices in $R_{12}$ coming from $K''_5\;=\;\;
\stackrel{2}{\bullet}\hspace{-2mm}\raisebox{.9mm}{\rule{6mm}{.2mm}}\hspace{-2mm}
\stackrel{1}{\bullet}\hspace{-2mm}\raisebox{.9mm}{\rule{6mm}{.2mm}}\hspace{-2mm}
\stackrel{4}{\bullet}\hspace{4mm} 
\stackrel{3}{\bullet}\hspace{-2mm}\raisebox{.9mm}{\rule{6mm}{.2mm}}\hspace{-2mm}
\stackrel{5}{\bullet}\hspace{-2mm}$}
\label{tabla4}
\end{table}

\begin{table}
\centerline{\begin{tabular}{|c|c|c|}
\hline $\dim=4$ & $\dim=5$ & $\dim=6$ \\ \hline
$13|14|23|24|25$ & $13|14|23|24|25|35$ & $13|14|23|24|25|35|45$ \\ 
$13|23|24|25|35$ & $13|14|23|24|25|45$ & \\ 
$14|23|24|25|45$ & $13|23|24|25|35|45$ & \\ 
 & $14|23|24|25|35|45$ & \\\hline
\end{tabular}}
\caption{Simplices in $R_{12}$ coming from $K''_5\;=\;\;
\stackrel{2}{\bullet}\hspace{-2mm}\raisebox{.9mm}{\rule{6mm}{.2mm}}\hspace{-2mm}
\stackrel{1}{\bullet}\hspace{-2mm}\raisebox{.9mm}{\rule{6mm}{.2mm}}\hspace{-2mm}
\stackrel{5}{\bullet}\hspace{4mm} 
\stackrel{3}{\bullet}\hspace{-2mm}\raisebox{.9mm}{\rule{6mm}{.2mm}}\hspace{-2mm}
\stackrel{4}{\bullet}\hspace{-2mm}$}
\label{tabla5}
\end{table}

\begin{table}
\centerline{\begin{tabular}{|c|c|c|c|}
\hline $\dim=3$ & $\dim=4$ & $\dim=5$ & $\dim=6$ \\ \hline
$13|14|15|24$ & $13|14|15|24|25$ & $13|14|15|24|25|34$ & $13|14|15|24|25|34|35$ \\ 
$13|14|15|25$ & $13|14|15|24|34$ & $13|14|15|24|25|35$ & \\ 
                      & $13|14|15|24|35$ & $13|14|15|24|34|35$ & \\ 
                      & $13|14|15|25|34$ & $13|14|15|25|34|35$ & \\ 
                      & $13|14|15|25|35$ &                                 & \\\hline
\end{tabular}}
\caption{Simplices in $R_{12}$ coming from $K''_5\;=\;\;
\stackrel{1}{\bullet}\hspace{-2mm}\raisebox{.9mm}{\rule{6mm}{.2mm}}\hspace{-2mm}
\stackrel{2}{\bullet}\hspace{-2mm}\raisebox{.9mm}{\rule{6mm}{.2mm}}\hspace{-2mm}
\stackrel{3}{\bullet}\hspace{4mm} 
\stackrel{4}{\bullet}\hspace{-2mm}\raisebox{.9mm}{\rule{6mm}{.2mm}}\hspace{-2mm}
\stackrel{5}{\bullet}\hspace{-2mm}$}
\label{tabla6}
\end{table}

\begin{table}
\centerline{\begin{tabular}{|c|c|c|c|}
\hline $\dim=3$ & $\dim=4$ & $\dim=5$ & $\dim=6$ \\ \hline
$13|14|15|23$ & $13|14|15|23|25$ & $13|14|15|23|25|34$ & $13|14|15|23|25|34|45$ \\ 
                      & $13|14|15|23|34$ & $13|14|15|23|25|45$ & \\ 
                      & $13|14|15|23|45$ & $13|14|15|23|34|45$ & \\ 
                      & $13|14|15|25|45$ & $13|14|15|25|34|45$ & \\\hline
\end{tabular}}
\caption{Simplices in $R_{12}$ coming from $K''_5\;=\;\;
\stackrel{1}{\bullet}\hspace{-2mm}\raisebox{.9mm}{\rule{6mm}{.2mm}}\hspace{-2mm}
\stackrel{2}{\bullet}\hspace{-2mm}\raisebox{.9mm}{\rule{6mm}{.2mm}}\hspace{-2mm}
\stackrel{4}{\bullet}\hspace{4mm} 
\stackrel{3}{\bullet}\hspace{-2mm}\raisebox{.9mm}{\rule{6mm}{.2mm}}\hspace{-2mm}
\stackrel{5}{\bullet}\hspace{-2mm}$}
\label{tabla7}
\end{table}

\begin{table}
\centerline{\begin{tabular}{|c|c|c|}
\hline $\dim=4$ & $\dim=5$ & $\dim=6$ \\ \hline
 $13|14|15|23|24$ & $13|14|15|23|24|35$ & $13|14|15|23|24|35|45$ \\ 
 $13|14|15|23|35$ & $13|14|15|23|24|45$ & \\ 
 $13|14|15|24|45$ & $13|14|15|23|35|45$ & \\ 
                            & $13|14|15|24|35|45$ & \\\hline
\end{tabular}}
\caption{Simplices in $R_{12}$ coming from $K''_5\;=\;\;
\stackrel{1}{\bullet}\hspace{-2mm}\raisebox{.9mm}{\rule{6mm}{.2mm}}\hspace{-2mm}
\stackrel{2}{\bullet}\hspace{-2mm}\raisebox{.9mm}{\rule{6mm}{.2mm}}\hspace{-2mm}
\stackrel{5}{\bullet}\hspace{4mm} 
\stackrel{3}{\bullet}\hspace{-2mm}\raisebox{.9mm}{\rule{6mm}{.2mm}}\hspace{-2mm}
\stackrel{4}{\bullet}\hspace{-2mm}$}
\label{tabla8}
\end{table}

\begin{table}
\centerline{\begin{tabular}{|c|c|}
\hline $\dim=5$ & $\dim=6$ \\ \hline
$13|14|15|23|24|25$ & $13|14|15|23|24|25|45$ \\ 
$13|14|15|24|25|45$ & \\ 
$14|15|23|24|25|45$ & \\\hline
\end{tabular}}
\caption{Simplices in $R_{12}$ coming from $K''_5\;=\;\;
\stackrel{4}{\bullet}\hspace{-2mm}\raisebox{.9mm}{\rule{6mm}{.2mm}}\hspace{-2mm}
\stackrel{3}{\bullet}\hspace{-2mm}\raisebox{.9mm}{\rule{6mm}{.2mm}}\hspace{-2mm}
\stackrel{5}{\bullet}\hspace{4mm} 
\stackrel{1}{\bullet}\hspace{-2mm}\raisebox{.9mm}{\rule{6mm}{.2mm}}\hspace{-2mm}
\stackrel{2}{\bullet}\hspace{-2mm}$}
\label{tabla9}
\end{table}

\begin{table}
\centerline{\begin{tabular}{|c|c|}
\hline $\dim=5$ & $\dim=6$ \\ \hline
 $13|14|15|23|24|34$ & $13|14|15|23|24|25|34$ \\ 
$13|14|23|24|25|34$ & \\\hline
\end{tabular}}
\caption{Simplices in $R_{12}$ coming from $K''_5\;=\;\;
\stackrel{3}{\bullet}\hspace{-2mm}\raisebox{.9mm}{\rule{6mm}{.2mm}}\hspace{-2mm}
\stackrel{4}{\bullet}\hspace{-2mm}\raisebox{.9mm}{\rule{6mm}{.2mm}}\hspace{-2mm}
\stackrel{5}{\bullet}\hspace{4mm} 
\stackrel{1}{\bullet}\hspace{-2mm}\raisebox{.9mm}{\rule{6mm}{.2mm}}\hspace{-2mm}
\stackrel{2}{\bullet}\hspace{-2mm}$}
\label{tabla10}
\end{table}

\begin{table}
\centerline{\begin{tabular}{|c|c|}
\hline $\dim=5$ & $\dim=6$ \\ \hline
 $13|14|15|23|25|35$ & $13|14|15|23|24|25|35$ \\  
 $13|15|23|24|25|35$ & \\\hline
\end{tabular}}
\caption{Simplices in $R_{12}$ coming from $K''_5\;=\;\;
\stackrel{3}{\bullet}\hspace{-2mm}\raisebox{.9mm}{\rule{6mm}{.2mm}}\hspace{-2mm}
\stackrel{5}{\bullet}\hspace{-2mm}\raisebox{.9mm}{\rule{6mm}{.2mm}}\hspace{-2mm}
\stackrel{4}{\bullet}\hspace{4mm} 
\stackrel{1}{\bullet}\hspace{-2mm}\raisebox{.9mm}{\rule{6mm}{.2mm}}\hspace{-2mm}
\stackrel{2}{\bullet}\hspace{-2mm}$}
\label{tabla11}
\end{table}

\begin{table}[h!]
\centerline{\begin{tabular}{|c|c|c|}
\hline $\dim=4$ & $\dim=5$ & $\dim=6$ \\ \hline
 $a=13|14|15|23|35$ & $e=13|14|15|23|24|35$ & $r=13|14|15|23|24|25|35$ \\  
 $b=13|14|15|24|45$ & $f=13|14|15|23|24|45$ & $s=13|14|15|23|24|25|45$ \\  
 $c=13|23|24|25|35$ & $g=13|14|15|23|25|35$ & $t=13|14|15|23|24|35|45$ \\  
 $d=14|23|24|25|45$ & $h=13|14|15|23|35|45$ & $u=13|14|23|24|25|35|45$ \\
 & $i=13|14|15|24|25|45$ & \\ 
 & $j=13|14|15|24|35|45$ & \\ 
 & $k=13|14|23|24|25|35$ & \\ 
 & $\ell=13|14|23|24|25|45$ & \\ 
 & $m=13|15|23|24|25|35$ & \\ 
 & $n=13|23|24|25|35|45$ & \\ 
 & $p=14|15|23|24|25|45$ & \\ 
 & $q=14|23|24|25|35|45$ & \\\hline
\end{tabular}}
\caption{Simplices in $R_{34}$}
\label{tabla12}
\end{table}

As discussed in Section~\ref{section3}, $X_{12}$ is a subcomplex of $D_{5,2}$ with an acyclic pairing that has a single critical cell (in dimension 0). So we can focus on constructing an acyclic pairing on $R_{12}$, i.e., the family of 86 simplices in Tables~\ref{tabla1}--\ref{tabla11}. The construction differs from the one we used in Section~\ref{section3} (where we defined a pairing $\mathcal{Q}$ whose acyclicity was proved by an inductive argument based on Proposition~\ref{relativized}). This time we apply one further round of an inclusion-exclusion pairing. Let:
\begin{itemize}
\item $\mathcal{P}\hspace{.3mm}'_{\hspace{-.3mm}34}$ be the (acyclic) matching on $D_{n,1}$ given by inclusion-exclusion of the edge $34$;
\item $\mathcal{P}_{34}$ be the restriction of $\mathcal{P}\hspace{.3mm}'_{\hspace{-.3mm}34}$ to $R_{12}$ ($\mathcal{P}_{34}$ is automatically acyclic);
\item $R_{34}:=R_{12}-X_{34}$, where $X_{34}$ is the family of graphs in $R_{12}$ with a $\mathcal{P}_{34}$-matching pair (of course, the matching pair should also lie in $R_{12}$). Table~\ref{tabla12} describes $R_{34}$.
\end{itemize}
The key point is that $X_{12}\cup X_{34}$ is a subcomplex of $D_{5,2}$. This follows from checking that no simplex in $R_{34}$ is a face of some simplex in $X_{34}$. So, a new application of Proposition~\ref{descomposicion} allows us to reduce further the problem, namely, we construct an acyclic matching $\mathcal{R}_{34}$ on $R_{34}$. As indicated in Table~\ref{tabla12}, $R_{34}$ consists of (4,12,4) simplices in dimension (4,5,6). The resulting face poset structure is simple enough (see Figure~\ref{mhd}) to construct, from scratch, the required acyclic matching. There are 16 different pairings in $R_{34}$ working for our purposes. One such instance is $\mathcal{R}_{34}:=\{(a,e),(b,f),(c,k),(d,\ell),(g,r),(i,s),(j,t),(n,u)\},$ which is acyclic and has only four critical simplices all in dimension 5, as can be easily seen from its modified Hasse diagram in Figure~\ref{mhd}. Putting everything together: $\mathcal{P}_{12}\cup\mathcal{P}_{34}\cup\mathcal{R}_{34}$ is an acyclic matching on $D_{5,2}$ with a single critical 0-cell (coming from $\mathcal{P}_{12}$) and four additional critical cells in dimension 5 (coming from $\mathcal{R}_{34}$). The result follows.
\end{proof}

\begin{figure}
$\begin{tikzpicture}[x=1.4cm,y=7cm]
\node at (0,0) {$e$};
\node at (1,0) {$f$};
\node at (2,0) {$g$};
\node at (3,0) {$h$};
\node at (4,0) {$i$};
\node at (5,0) {$j$};
\node at (6,0) {$k$};
\node at (7,0) {$\ell$};
\node at (8,0) {$m$};
\node at (9,0) {$n$};
\node at (10,0) {$p$};
\node at (11,0) {$q$};
\node at (1.5,-.5) {$a$};
\node at (3.5,-.5) {$b$};
\node at (7.5,-.5) {$c$};
\node at (9.5,-.5) {$d$};
\node at (1.5,1) {$r$};
\node at (3,1) {$s$};
\node at (4.5,1) {$t$};
\node at (9.5,1) {$u$};
\draw[very thick](0,-.1)--(1.5,-.42);
\draw[very thick](1,-.1)--(3.5,-.42);
\draw[very thin](2,-.1)--(1.5,-.42);
\draw[very thin](3,-.1)--(1.5,-.42);
\draw[very thin](4,-.1)--(3.5,-.42);
\draw[very thin](5,-.1)--(3.5,-.42);
\draw[very thick](6,-.1)--(7.5,-.42);
\draw[very thick](7,-.1)--(9.5,-.42);
\draw[very thin](8,-.1)--(7.5,-.42);
\draw[very thin](9,-.1)--(7.5,-.42);
\draw[very thin](10,-.1)--(9.5,-.42);
\draw[very thin](11,-.1)--(9.5,-.42);
\draw[very thin](0,.1)--(1.5,.92);
\draw[very thin](0,.1)--(4.5,.92);
\draw[very thin](1,.1)--(3,.92);
\draw[very thin](1,.1)--(4.5,.92);
\draw[very thick](2,.1)--(1.5,.92);
\draw[very thin](3,.1)--(4.5,.92);
\draw[very thick](4,.1)--(3,.92);
\draw[very thick](5,.1)--(4.5,.92);
\draw[very thin](6,.1)--(1.5,.92);
\draw[very thin](6,.1)--(9.5,.92);
\draw[very thin](7,.1)--(3,.92);
\draw[very thin](7,.1)--(9.5,.92);
\draw[very thin](8,.1)--(1.5,.92);
\draw[very thick](9,.1)--(9.5,.92);
\draw[very thin](10,.1)--(3,.92);
\draw[very thin](11,.1)--(9.5,.92);
\end{tikzpicture}$
\caption{Modified Hasse diagram for $\mathcal{R}_{34}$ with thick lines standing for matched pairs}
\label{mhd}
\end{figure}
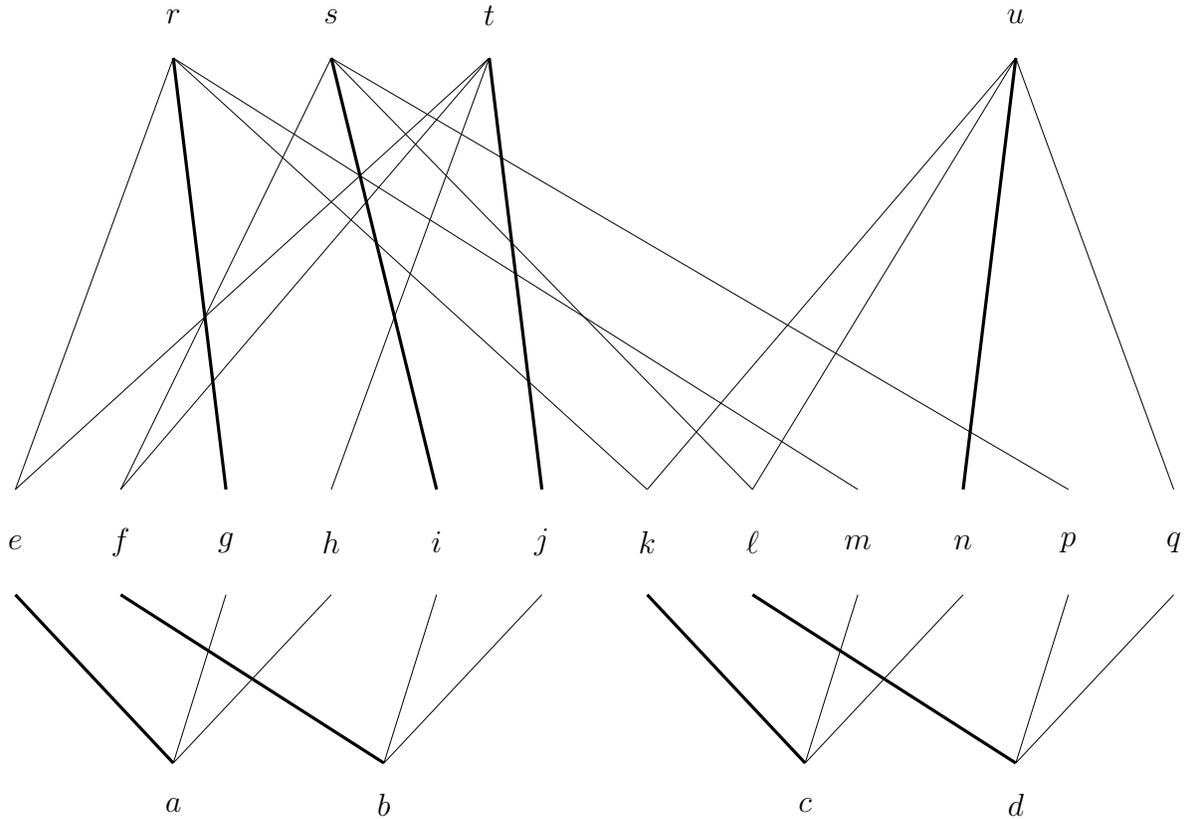


\smallskip\bigskip
{\small \sc Departamento de Matem\'aticas

Centro de Investigaci\'on y de Estudios Avanzados del I.P.N.

Av.~Instituto Polit\'ecnico Nacional n\'umero 2508

San Pedro Zacatenco, M\'exico City 07000, M\'exico

{\tt jesus@math.cinvestav.mx}

{\tt  idskjen@math.cinvestav.mx}}
\end{document}